\documentclass[11pt,a4paper]{article}

\usepackage{fullpage}

\usepackage[round]{natbib}

\usepackage{amsmath}
\usepackage{amsthm}
\usepackage{ascmac}
\usepackage{amsfonts}
\usepackage[dvipdfmx]{graphicx}
\allowdisplaybreaks[1] 

\newcommand{\dd}{{\mathrm d}}

\newcommand{\argmax}{\mathop{\rm argmax}\limits}
\newcommand{\argmin}{\mathop{\rm argmin}\limits}
\newcommand{\E}{\mathrm{E}}
\newcommand{\esssup}{{\mathrm{ess}\sup}}

\newtheorem{theorem}{Theorem}
\newtheorem{lemma}{Lemma}

\newtheorem{example}{Example}

\newtheorem{remark}{Remark}

\title{Relations Between the Conditional Normalized Maximum Likelihood Distributions and the Latent Information Priors}

\author{Mutsuki KOJIMA${}^*$ and Fumiyasu KOMAKI${}^*{}^{\dagger}$\\ \\
        ${}^*$Department of Mathematical Informatics\\
        Graduate School of Information Science and Technology\\
        The University of Tokyo, Tokyo, Japan\\
        ${}^{\dagger}$RIKEN Brain Science Institute, Wako-shi, Japan\\
        \texttt{\normalsize \{mutsuki\_kojima,komaki\}@mist.i.u-tokyo.ac.jp}}
\date{}

\begin{document}
\maketitle

\begin{abstract}
We reveal the relations between the conditional normalized maximum likelihood (CNML) distributions and Bayesian predictive densities based on the latent information priors (LIPs). In particular, CNML3, which is one type of CNML distributions, is investigated. 
The Bayes projection of a predictive density, which is an information projection of the predictive density
on a set of Bayesian predictive densities, is considered. We prove that the sum of the Bayes projection divergence of CNML3 and the conditional mutual information is asymptotically constant. This result implies that the Bayes projection of CNML3 (BPCNML3) is asymptotically identical to the Bayesian predictive density based on LIP. In addition, under some stronger assumptions, we show that BPCNML3 exactly coincides with the Bayesian predictive density based on LIP.
\end{abstract}


\section{Introduction}
We construct predictive densities for future variables based on observed data.  Let $(X,\mathcal{F})$ be a measurable space and let $\mathcal{M}=\{p(x|\theta)|x\in X, \theta\in\Theta\subset \mathbf{R}^d\}$ be a statistical model, where $p(x|\theta)$ is the probability density function with respect to a $\sigma$-finite measure $\mu$ on $(X,\mathcal{F})$. We assume that observations $x^N:= (x_1,\ldots ,x_N)^{\top} \in X^N$ and future variables $y^M:= (y_1,\ldots, y_M)^{\top}\in X^M$ are independent and identically distributed random variables with probability distribution $\mathcal{M}$. Thus, the joint probability density function of $x^N$ and $y^M$ is 
\[
p(x^N,y^M|\theta)=\prod_{i=1}^Np(x_i|\theta)\prod_{j=1}^Mp(y_j|\theta).
\]

A predictive density $q(y^M|x^N)$ is a conditional probability density, i.e., a function from $X^N\times X^M$ to $\mathbf{R}_+$ satisfying $\int_{X^M} \dd\mu(y^M) q(y^M|x^N)=1$. The goodness of prediction fit of $q(y^M|x^N)$ is evaluated by the average Kullback--Leibler divergence (simply referred to as {\it KL risk} in this paper) :
\[
R_{\mathrm{KL}}^{N,M}(\theta, q):= \int_{X^N} \dd\mu(x^N) p(x^N|\theta) \int_{X^M} \dd \mu(y^M) p(y^M|\theta) \log\frac{p(y^M|\theta)}{q(y^M|x^N)}.
\]

In information theory, the Bayes risk
\[
R_{\mathrm{KL}}^{N,M}(\pi,p_{\pi}):=\int_{\Theta} \dd\pi (\theta) R_{\mathrm{KL}}^{N,M}(\theta, p_{\pi}),
\]
is called conditional mutual information when $N>0$ \citep{CoverThomas2006}. 
Latent information priors (LIPs) are defined as prior distributions on $\Theta$
that maximize the conditional mutual information, see \cite{Komaki2011}.
Bayesian predictive densities based on LIPs are minimax predictive densities under KL risk when $\mathcal{M}$
is a submodel of the multinomial distribution.
The LIPs are different from Jeffreys priors in general.
In addition, when $N>0$ and the model is a joint location and scale model,
we note that the minimax predictive densities under KL risk do not have to
match the Bayesian predictive densities based on Jeffreys priors as shown by \cite{LiangBarron2004}. 

On the other hand, in the context of information-theoretic learning, the normalized maximum likelihood (NML) distributions, introduced by \cite{Shtarkov1987}, are important predictive densities with no observation ($N=0$). The NML distribution is defined by
\[
q_{\mathrm{NML}}(y^M):=\frac{p(y^M|\hat{\theta}(y^M))}{\int_{X^M} \dd\mu(z^M) p(z^M|\hat{\theta}(z^M))},
\]
where $\hat{\theta}(z^M):=\mathrm{argmax}_{\theta}~p(z^M|\theta)$. \cite{Shtarkov1987} showed that the NML distribution achieves the minimax regret:
\begin{align*}
q_{\mathrm{NML}}=\argmin_{q}\max_{y^M}\{-\log q(y^M)-(-\log p(y^M|\hat{\theta}(y^M)))\}.
\end{align*}
However, NML distributions have a serious problem that the normalizing constants diverge to infinity even if $\mathcal{M}$ is a simple statistical model such as the normal, Poisson, or geometric distribution. To remedy the problem, \cite{Grunwald2007} proposed three types of generalizations of NML distributions called conditional normalized maximum likelihood (CNML) distributions:
\begin{align*}
q_{\mathrm{CNML1}}(y^M|x^N)&:=\frac{p(x^N, y^M|\hat{\theta}(y^M))}{\int_{X^M} \dd\mu(z^M) p(x^N, z^M|\hat{\theta}(z^M))},\\
q_{\mathrm{CNML2}}(y^M|x^N)&:=\frac{p(x^N, y^M|\hat{\theta}(x^N, y^M))}{\int_{X^M} \dd\mu(z^M) p(x^N, z^M|\hat{\theta}(x^N, z^M))},\\
q_{\mathrm{CNML3}}(y^M|x^N)&:=\frac{p(y^M|x^N, \hat{\theta}(x^N, y^M))}{\int_{X^M} \dd\mu(z^M) p(z^M|x^N, \hat{\theta}(x^N, z^M))},
\end{align*}
where $\hat{\theta}(x^N,z^M):=\mathrm{argmax}_{\theta}~p(x^N,z^M|\theta)$.
By conditioning on observations $x^N$, the normalizing constants of CNML distributions do not diverge to infinity, and the distributions are defined as predictive densities with some observations ($N>0$). As with the NML distribution, CNML-$i$ ($i=1,2,3$) achieves the minimax conditional regret-$i$ ($i=1,2,3$):
\begin{align*}
q_{\mathrm{CNML1}}&=\argmin_{q}\max_{y^M}\{-\log q(y^M|x^N)-(-\log p(x^N,y^M|\hat{\theta}(y^M)))\},\\
q_{\mathrm{CNML2}}&=\argmin_{q}\max_{y^M}\{-\log q(y^M|x^N)-(-\log p(x^N,y^M|\hat{\theta}(x^N,y^M)))\},\\
q_{\mathrm{CNML3}}&=\argmin_{q}\max_{y^M}\{-\log q(y^M|x^N)-(-\log p(y^M|x^N,\hat{\theta}(x^N,y^M)))\}.
\end{align*}

Our results are twofold. First, we show that the sum of the Bayes projection divergence of CNML3 and the conditional mutual information is asymptotically constant. The Bayes projection of a predictive density is an information projection,
a generalization of the information projection studied by \cite{Csiszar1975}, of the predictive density
on a set of Bayesian predictive densities (see Section \ref{preliminaries}). Throughout the paper, ``asymptotic'' means that the number of observations, $N$, is fixed, and the number of future variables, $M$, goes to infinity. Roughly speaking, the first result implies that the Bayes projection of CNML3 (BPCNML3) is asymptotically identical to the Bayesian predictive density based on LIP. Second, under some stronger assumptions, we show that the BPCNML3 exactly coincides with the Bayesian predictive density based on LIP. These results indicate that CNML3 is related to LIPs. 

Among CNML distributions, CNML2 has received much attention \citep{KotlowskiGrunwald2011,Hedayati2012a, Hedayati2012b, BartlettEtAl2013,Harremoes2013}, and it has been recognized as the only natural generalization of NML distributions \citep{Grunwald2012}. \cite{Grunwald2007} showed that CNML1 and CNML2 are asymptotically equal to the Bayesian predictive density based on Jeffreys prior. Under some regularity conditions, \cite{Hedayati2012a} showed that CNML2 is identical to the Bayesian predictive density based on Jeffreys prior even when $M$ is finite. Because of the connection with Jeffreys prior, CNML2 is considered to be the most important predictive density among CNML distributions. 

However, we argue that LIPs, not Jeffreys priors, are naturally related to minimax predictive densities under the conditional regret when $N>0$. The reason is as follows. The regret and Kullback--Leibler divergence are widely known to be naturally related in the sense that they are special versions of the R\'{e}nyi divergence \citep{Renyi1961, ErvenHarremoes2014}. Notably, when $N=0$ and statistical model $\mathcal{M}$ is the multinomial distribution, \cite{XieBarron2000} showed that a Bayesian predictive density based on a modification of Jeffreys prior asymptotically achieves the minimax regret. When $N=0$ and the model satisfies some regularity conditions, \cite{ClarkeBarron1994} showed that Jeffreys prior is asymptotically least favorable under KL risk. Roughly speaking, when $N=0$, Bayesian predictive densities based on Jeffreys priors are asymptotically minimax under both the regret and KL risk. In addition, the NML distribution is known to asymptotically coincide with the Bayesian predictive density based on Jeffreys prior \citep{Grunwald2007}. These studies imply that least favorable priors under KL risk are connected with minimax predictive densities under the regret when $N=0$. Therefore, as is the case for $N=0$, we insist that LIPs are naturally related to minimax predictive densities under the conditional regret because LIPs are least favorable priors under KL risk when $N>0$. 

Our results shed light on the connection between LIPs and CNML3. Although CNML2 has received the most attention among CNML distributions, we consider that CNML3, not CNML2, is more in line with the minimax KL risk approach and is the most important predictive density among CNML distributions. Notably, \cite{Grunwald2007} also vaguely suggested that CNML3 is more in line with the minimax KL risk approach (called Liang and Barron's approach \citep{LiangBarron2004} in his book \citep{Grunwald2007}) than CNML1 and CNML2.

The remainder of this paper is organized as follows. In Section \ref{preliminaries}, we define the Bayes projection of predictive densities and review the definition and properties of LIPs. In Section \ref{main_results}, we state the main results. In Section \ref{numerical_experiments}, we confirm that the main results hold for the binomial distributions through numerical experiments. In Section \ref{conclusion}, we conclude our study.

\section{Preliminaries}\label{preliminaries}
Let $K$ be a compact set of $\Theta$ and $\mathcal{P}_K$ be the set of  all probability measures on $\Theta$ whose support sets are contained in $K$. We assume that $\mathcal{P}_K$ is endowed with the weak convergence topology and the corresponding Borel sigma algebra. By the Prokhorov  theorem, $\mathcal{P}_K$ is compact.

\subsection{Bayes Projection of Predictive Densities}
We define the projection of predictive densities on a set of Bayesian predictive densities. Let $D_{K,q}^{N,M}(\pi)$ be a divergence from Bayesian predictive density based on $\pi$ to predictive density $q$:
\[
D_{K,q}^{N,M}(\pi):= \int_{X^N\times X^M} \dd\mu(x^N,y^M) p_{\pi}(x^N,y^M) \log \frac{p_{\pi}(x^N,y^M)}{q(y^M|x^N)p_{\pi}(x^N)},\quad \pi\in \mathcal{P}_{K},
\]
where
\[
p_{\pi}(x^k):=\int_{\Theta} \dd\pi(\theta) p(x^k|\theta) .
\]
Divergence $D_{K,q}^{N,M}$ is convex with respect to $\pi$. Let  $\pi_1$ and $\pi_2$ in $\mathcal{P}_K$ and $w\in (0,1)$. We define $\pi_w:=w\pi_1+(1-w)\pi_2$. By the log sum inequality,
\begin{align*}
p_{\pi_w}&(x^N,y^M) \log \frac{p_{\pi_w}(x^N,y^M)}{q(y^M|x^N)p_{\pi_w}(x^N)}\\
\leq {}& wp_{\pi_1}(x^N,y^M)\log\frac{p_{\pi_1}(x^N,y^M)}{q(y^M|x^N)p_{\pi_1}(x^N)}+(1-w)p_{\pi_2}(x^N,y^M)\log\frac{p_{\pi_2}(x^N,y^M)}{q(y^M|x^N)p_{\pi_2}(x^N)}.
\end{align*}
Therefore,
\begin{align*}
D_{K,q}^{N,M}(w\pi_1+(1-w)\pi_2)\leq wD_{K,q}^{N,M}(\pi_1)+(1-w)D_{K,q}^{N,M}(\pi_2),\quad w \in (0,1).
\end{align*}
Since $\mathcal{P}_K$ is compact, if map $\mathcal{P}_K\ni \pi \mapsto D_{K,q}^{N,M}(\pi)\in \mathbf{R}$ is strictly convex and lower semicontinuous, then there exists unique minimizer $\hat{\pi}^{N,M}_{K,q}\in\mathcal{P}_K$ such that 
\[
D_{K,q}^{N,M}(\hat{\pi}^{N,M}_{K,q})=\inf_{\pi\in\mathcal{P}_K} D_{K,q}^{N,M}(\pi).
\]
We refer to the Bayesian predictive density based on $\hat{\pi}^{N,M}_{K,q}$ as {\it Bayes projection} of $q$. 

\cite{Komaki2011} showed that KL risk of the Bayes projection of $q$ is not larger than that of $q$ if the statistical model is a submodel of the multinomial distribution.

\subsection{Latent Information Priors}
In information theory, the Bayes risk
\[
R_{\mathrm{KL}}^{N,M}(\pi,p_{\pi}):=\int_{\Theta} \dd\pi (\theta) R_{\mathrm{KL}}^{N,M}(\theta, p_{\pi}),
\]
is called mutual information when $N=0$ and conditional mutual information when $N>0$ \citep{CoverThomas2006}.  The conditional mutual information is concave with respect to $\pi\in\mathcal{P}_K$. LIPs are defined as priors that maximize the conditional mutual information:
\[
\hat{\pi}_{K,\mathrm{LIP}}^{N,M}:=\argmax_{\pi\in\mathcal{P}_K} R_{\mathrm{KL}}^{N,M}(\pi,p_{\pi}).
\]
Since $\mathcal{P}_K$ is compact, if map $\mathcal{P}_K\ni \pi \mapsto R_{\mathrm{KL}}^{N,M}(\pi,p_{\pi})\in \mathbf{R}$ is strictly concave and upper semicontinuous, then $\hat{\pi}_{K,\mathrm{LIP}}^{N,M}$ is the unique maximizer.

Because LIPs are the least favorable priors \citep{Ferguson1967}, the Bayesian predictive densities based on LIPs are naturally related to minimax predictive densities under KL risk. Notably, \cite{Komaki2011} showed that Bayesian predictive densities based on LIPs are minimax predictive densities under KL risk when $\mathcal{M}$ is a submodel of the multinomial distribution.

\section{Main Results}\label{main_results}

Before showing the main results, we give basic assumptions and notations.

We assume that a maximum likelihood estimator (MLE) $\hat{\theta}(z^k)\in\Theta$ exists for all $k\in\mathbf{N}$ and $z^k\in X^k$. 
We take a compact set $K$ contained in the interior of $\Theta$ such that $p(z|\theta)$ is strictly positive for all $z\in X$ and $\theta\in K$ and take a positive constant $\delta$ such that $K_{\delta}=\{\tilde{\theta}\in\Theta |\exists{\theta}\in K~\mathrm{s.t.}~|\theta-\tilde{\theta}|\leq \delta\}$ is also contained in the interior of $\Theta$. Here, $|\theta|$ denotes the Euclidean norm. We denote probabilities of events and expectations of random variables by $P_{\theta}(\cdot)$ and $\E_{\theta}(\cdot)$, respectively.

We state conditions and lemmas required to prove the main results. 

\begin{description}
\item[A1.] For all $z\in X$, the log-likelihood function $\log p(z|\theta)$ is Lipschitz continuous on $K_{\delta}$, i.e., there exists a measurable map  $L_{K_{\delta}}:X\to \mathbf{R}_+$ and $1\leq p\leq \infty$ such that for all $\theta_1,\theta_2\in K_{\delta}$
\begin{align*}
\big|\log p(z|\theta_1)- \log p(z|\theta_2)\big|\leq L_{K_{\delta}}(z)|\theta_1-\theta_2|,
\end{align*}
where $L_{K_{\delta}}(\cdot)$ satisfies
\begin{align*}
\sup_{\theta\in K} \int_{X}\dd\mu (z)p(z|\theta) \{ L_{K_{\delta}}(z)\}^p < \infty.
\end{align*}
We define $\{L_{K_{\delta}}\}^{\infty}:=\esssup_{z\in X}L_{K_{\delta}}(z)$.
	
\item[A2.]
\begin{align*}
\lim_{k\to\infty} \sup_{\theta\in K} \int_{X^k} \dd\mu(z^k) p(z^k|\theta) \big|\hat{\theta}(z^k)-\theta\big|^q=0,
\end{align*}
where $q \geq 1$ satisfies $1/p + 1/q=1$ ($q=\infty$ when $p=1$ and $q=1$ when $p=\infty$).
	
\item[A3.] There exists a measurable map  $T_K:X\to \mathbf{R}_+$ and $1 < r \leq \infty$ such that
\[
\sup_{\theta\in K}\{\log p(z|\hat{\theta}(z)) - \log p(z|\theta)\} \leq T_K(z),
\]
and $T_K$ satisfies
\[
\sup_{\theta\in K}\int_X \dd\mu(z) p(z|\theta) \{T_K(z)\}^r  < \infty.
\]
	
\item[A4.]
\[
\lim_{k\to\infty} \sup_{\theta\in K} \bigg|\int_{X^k} \dd\mu(z^k) p(z^k|\theta) \log \frac{p(z^k|\hat{\theta}(z^k))}{p(z^k|\theta)} -\frac{d}{2} \bigg|=0.
\]
	
\item[A5.] There exist constants $C^{N,M}$ that do not depend on $\theta$ such that
\[
\lim_{M\to\infty} \sup_{\theta\in K} \bigg| \int_{X^N} \dd\mu(x^N) p(x^N|\theta) \log \bigg(\int_{X^M} \dd\mu(y^M) p(y^M|\hat{\theta}(x^N,y^M))\bigg) - C^{N,M}\bigg|=0.
\]
\end{description}

\begin{remark}
\upshape
The integrand in condition A4 is known as the likelihood ratio statistic. The likelihood ratio statistic is widely known to converge in distribution to the chi-squared distribution with degrees of freedom $d/2$ under some mild conditions \citep{Wilks1938}. Because the mean of the chi-squared distribution is $d/2$, condition A4 is considered to be satisfied for many regular statistical models. However,  except for \cite{ClarkeBarron1989}, we are not aware of studies about conditions on the $L^1$ convergence of the likelihood ratio statistic.
\end{remark}

\begin{lemma}\label{sufficient_conditions_for_uniform_consistency}
Assume that statistical model $\mathcal{M}$ satisfies condition A2. Then,
\[
\lim_{k\to\infty} \sup_{\theta\in K} P_{\theta}\bigg(\bigg\{\hat{\theta}(z^k)\not\in K_{\delta}\bigg\} \bigg)=0.
\]
\end{lemma}
\begin{proof}
By the Markov and H\"{o}lder inequalities, for all $\theta\in K$,
\[
P_{\theta}\bigg(\bigg\{\hat{\theta}(z^k)\not\in K_{\delta}\bigg\} \bigg) \leq
P_{\theta}(|\hat{\theta}(z^k)-\theta|>\delta) \leq \frac{1}{\delta} \bigg\{\E_{\theta}(|\hat{\theta}(z^k)-\theta|^q)\bigg\}^{\frac{1}{q}}.
\]
Since condition A2 is satisfied, the claim is verified. 
\end{proof}

\begin{lemma}\label{lemma_approx_bias_of_loglikelihood}
Assume that conditions A1--A4 are satisfied. Then,
\[
\lim_{M\to\infty} \sup_{\theta\in K} \bigg|\int_{X^N\times X^M} \dd\mu(x^N,y^M)~ p(x^N, y^M|\theta) \log \frac{p(y^M|\theta)}{p(y^M|\hat{\theta}(x^N, y^M))}+\frac{d}{2}\bigg|= 0.
\]
\end{lemma}
\begin{proof}
See Appendix. 
\end{proof}

We state our first result. 

\begin{theorem}\label{asymptotic_evaluation_of_projection_dist_CNML3}
Let $K$ be a compact set that is contained in the interior of  $\Theta$ and assume that $p(z|\theta)$ is strictly positive for all $z\in X$ and $\theta\in K$. Assume also that conditions A1--A5 are satisfied. 
	
Then,
\begin{align}\label{asymptotic_dist_and_KLrisk_from_cnml3_to_Bayes}
\lim_{M\to\infty} \sup_{\pi\in\mathcal{P}_K} \big|D_{K,q_{\mathrm{CNML3}}}^{N,M}(\pi)+R_{\mathrm{KL}}^{N,M}(\pi, p_{\pi})-\tilde{C}^{N,M}\big|=0,
\end{align}
where $\tilde{C}^{N,M}=C^{N,M}-d/2$ that does not depend on the choice of $\pi$.
\end{theorem}

By deforming \eqref{asymptotic_dist_and_KLrisk_from_cnml3_to_Bayes}, we have
\begin{align}\label{deformation_asymptotic_dist_and_KLrisk_from_cnml3_to_Bayes}
D_{K,q_{\mathrm{CNML3}}}^{N,M}(\pi)=-R_{\mathrm{KL}}^{N,M}(\pi, p_{\pi})+\tilde{C}^{N,M}+o(1),
\end{align}
where term $o(1)$ satisfies $\lim_{M\to\infty}\sup_{\pi\in\mathcal{P}_K} |o(1)|=0$.

Asymptotically, in the right-hand side of \eqref{deformation_asymptotic_dist_and_KLrisk_from_cnml3_to_Bayes}, only the first term $R_{\mathrm{KL}}^{N,M}(\pi, p_{\pi})$ depends on the choice of $\pi$. Therefore, the LIP that maximizes $R_{\mathrm{ KL}}^{N,M}(\pi, p_{\pi})$ with respect to $\pi\in\mathcal{P}_K$ asymptotically coincides with the minimizer of the left-hand side of \eqref{deformation_asymptotic_dist_and_KLrisk_from_cnml3_to_Bayes}, i.e., $\hat{\pi}^{N,M}_{K,q_{\mathrm{CNML3}}}$. In other words, roughly speaking, BPCNML3 is asymptotically identical to the Bayesian predictive density based on the LIP. Notably, BPCNML3 is different from CNML3. Later, under some stronger conditions, we will show that BPCNML3 exactly coincides with the Bayesian predictive density based on the LIP even when $M$ is finite (see Theorem \ref{non-asymptotic_evaluation_of_projection_dist_CNML3}).

\begin{proof}[Proof of Theorem \ref{asymptotic_evaluation_of_projection_dist_CNML3}]
\begin{align*}
D_{K,q_{\mathrm{CNML3}}}^{N,M}(\pi)={}&\int_{\Theta\times X^N\times X^M} \dd\pi(\theta) \dd\mu(x^N,y^M) p(x^N,y^M|\theta)\log \frac{p_{\pi}(y^M|x^N)}{q_{\mathrm{CNML3}}(y^M|x^N)}\\
={}&-\int_{\Theta\times X^N\times X^M} \dd\pi(\theta) \dd\mu(x^N,y^M) p(x^N,y^M|\theta)\log \frac{p(y^M|\theta)}{p_{\pi}(y^M|x^N)}\\
&+\int_{\Theta\times X^N\times X^M} \dd\pi(\theta) \dd\mu(x^N,y^M) p(x^N,y^M|\theta)\log \frac{p(y^M|\theta)}{q_{\mathrm{CNML3}}(y^M|x^N)}.
\end{align*}
The first term is $-R_{\mathrm{KL}}^{N,M}(\pi,p_{\pi})$. The second term is decomposed as
\begin{align*}
\int_{\Theta\times X^N\times X^M}& \dd\pi(\theta) \dd\mu(x^N,y^M) p(x^N,y^M|\theta)\log \frac{p(y^M|\theta)}{q_{\mathrm{CNML3}}(y^M|x^N)}\\
={}& \int_{\Theta} \dd\pi(\theta) \bigg\{\int_{X^N\times X^M}\dd\mu(x^N,y^M)
p(x^N,y^M|\theta) \log \frac{p(y^M|\theta)}{p(y^M|\hat{\theta}(x^N,y^M))}+\frac{d}{2}\bigg\}\\
&+ \int_{\Theta} \dd\pi(\theta) \bigg\{\int_{X^N} \dd\mu(x^N)~p(x^N|\theta)
\log \bigg(\int_{X^M} \dd\mu(z^M) p(z^M|\hat{\theta}(x^N,z^M))\bigg)-C^{N,M}\bigg\}\\
&+ C^{N,M}-\frac{d}{2}.
\end{align*}
By Lemma \ref{lemma_approx_bias_of_loglikelihood} and assumption A5, we have
\begin{align*}
\int_{\Theta\times X^N\times X^M} \dd\pi(\theta) \dd\mu(x^N,y^M) p(x^N,y^M|\theta)\log \frac{p(y^M|\theta)}{q_{\mathrm{CNML3}}(y^M|x^N)}=C^{N,M}-\frac{d}{2} + o(1),
\end{align*}
where term $o(1)$ satisfies $\lim_{M\to\infty}\sup_{\pi\in\mathcal{P}_K} |o(1)|=0$. Therefore, the claim is verified. 
\end{proof}

We give some examples that satisfy conditions A1--A5.

\begin{example}[Multinomial Distributions]\label{multinomial_dist}
\upshape
The first example is the multinomial distribution. Let $X=\{0,1,\ldots, d\}$ and $\Theta=\{(p_1,\ldots, p_d)| 0\leq p_i \leq 1~(i=1,\ldots, d), ~\sum_{i=1}^d p_i \leq 1\}$. We take a compact set $K$ that is contained in the interior of $\Theta$:
\[
K\subset \bigg\{\theta=(p_1,\ldots, p_d)|0<p_i<1~(i=1,\ldots, d),~\sum_{i=1}^d p_i <1\bigg\}.
\]
Since $K$ is contained in the interior of $\Theta$, we can find $\delta>0$ such that compact set $K_{\delta}$ is also in the interior of $\Theta$.
	
The probability function is 
\[
p(z|\theta) = \prod_{i=0}^d p_i^{z^{(i)}} , \quad z=(z^{(0)},\ldots, z^{(d)})^{\top}\in \{0,1\}^{d+1}, \quad p_0:=1-\sum_{i=1}^dp_i,
\]
where we identify elements in $X$ with $z=(z^{(0)},\ldots, z^{(d)})^{\top}\in \{0,1\}^{d+1}$ satisfying $\sum_{i=0}^d z^{(i)}=1$. Since there exists a positive constant $c_{K}$ such that $\inf_{\theta\in K}\min_{i=0,1,\ldots,d} p_i \geq c_{K}>0$,
\begin{align*}
\sup_{\theta\in K}\{\log p(z|\hat{\theta}(z)) - \log p(z|\theta)\}\leq \log 1-\inf_{\theta\in K}\log p(z|\theta)\leq -\log c_{K}.
\end{align*}
	
Similarly, there exists a positive constant $c_{K_{\delta}}>0$ such that $\inf_{\theta\in K_{\delta}}\min_{i=0,1,\ldots,d} p_i \geq c_{K_{\delta}}$. By the mean value theorem, for all $\theta_1,\theta_2\in K_{\delta}$ and $z\in X$,
\begin{align*}
\big|\log p(z|\theta_1)-\log p(z|\theta_2)\big| \leq \frac{1}{c_{K_{\delta}}} |\theta_1-\theta_2|.
\end{align*}
Therefore, condition A1 and A3 with any $p\in [1, \infty]$ and $r=\infty$ are satisfied. The MLE of the multinomial distribution is 
\[
\hat{\theta}(z^n)=\bigg(\frac{\sum_{i=1}^n z_i^{(1)}}{n},\ldots, \frac{\sum_{i=1}^n z_i^{(d)}}{n}\bigg),\quad z^n \in X^n,
\]
and the variance of the MLE is
\[
\E_{\theta} [|\hat{\theta}(z^n)-\theta|^2]=\frac{1}{n}\sum_{j=1}^d [p_j(1-p_j)].
\]
Hence, condition A2 with $q=2$ is satisfied. Concerning
conditions A4 and A5, we show two lemmas.
	
\begin{lemma}\label{lemma_multinomial_expectation_of_likelihoodratio}
For the multinomial distributions, condition A4 is satisfied.
\end{lemma}
\begin{proof}
Let $G_n$ be the likelihood ratio statistic:
\[
G_n(z^n;\theta):=\log\frac{p(z^n|\hat{\theta}(z^n))}{p(z^n|\theta)}.
\]
\cite{SmithEtAl1981} showed that for $\theta$ in the interior of $\Theta$,
\[
\E_{\theta}(G_n(z^n;\theta))=\frac{d}{2}+R_n(\theta),
\]
where $R_n$ satisfies
\[
|R_n(\theta)|\leq \sum_{j=0}^d np_j \frac{1}{6n^3p_j^3} \bigg|\E_{\theta}\bigg(\frac{\sum_{i=1}^n{z_i^{(j)}}}{n}-p_j\bigg)^3\bigg|=\frac{1}{n}\sum_{j=0}^d \frac{|(1-p_j)(1-2p_j)|}{6p_j}.
\]
Thus, $\lim_{n\to\infty}\sup_{\theta\in K}|R_n(\theta)|=0$. Consequently, the claim is verified.
\end{proof}
	
\begin{lemma}\label{lemma_multinomial_normconstants_of_CNML3}
For the multinomial distributions, the normalizing constant of CNML3 is independent of $x^N$. 
Therefore, condition A5 is satisfied.
\end{lemma}
\begin{proof}
See Appendix. 
\end{proof}
	
In conclusion, the multinomial distributions satisfy conditions A1-A5.
\end{example}

\begin{example}[Normal Distributions with Restricted Mean]\label{restricted_normal}
\upshape

We fix positive numbers $a$ and $b$ such that $a>b>0$. Let $\Theta=[-a,a]$ and $K=[-b,b]$. Since $a$ is strictly larger than $b$, we can take a positive constant $\delta$ satisfying $\delta < a-b$ and $K_{\delta}=[-b-\delta,b+\delta]\subset (-a,a)$.
	
We consider the normal distribution with mean $\theta\in\Theta$ and variance $1$. The probability density function is
\[
p(z|\theta)=\frac{1}{\sqrt{2\pi}}\exp\bigg(-\frac{(z-\theta)^2}{2}\bigg),\quad z\in X.
\]
For $\theta_1,\theta_2\in K_{\delta}$,  the log-likelihood function satisfies
\[
|\log p(z|\theta_1)-\log p(z|\theta_2)|\leq (|z|+a)|\theta_1-\theta_2|.
\]
Therefore, condition A1 is satisfied with $p=2$.
	
The MLE is 
\begin{align*}
\hat{\theta}(z^k)=\left\{
\begin{array}{ll}
-a, & \mathrm{if}~ \overline{z^k}<-a, \\
a, &  \mathrm{if}~ \overline{z^k}>a, \\
\overline{z^k}, & \mathrm{otherwise},
\end{array}
\right.
\end{align*}
where $\overline{z^k}:=\sum_{i=1}^kz_i/k$. We denote the probability density function of the one-dimensional normal distribution with mean $\mu$ and variance $\sigma^2$ by $\phi(z;\mu,\sigma^2)$. Since $\overline{z^k}$ is normally distributed with mean $\theta$ and variance $1/k$, 
\begin{align*}
\E_{\theta}&(\hat{\theta}(z^k)-\theta)^2\\
= {}&\int_{-\infty}^{-a} \dd z~(-a-\theta)^2\phi(z;\theta,1/k)+\int_a^{\infty}\dd z~ (a-\theta)^2\phi(z;\theta,1/k)
+\int_{-a}^a \dd z~(z-\theta)^2\phi(z;\theta,1/k) \\
\leq {}& 4a^2 \int_{-\infty}^{\sqrt{k}(-a-\theta)}\dd z~\phi(z;0,1)
+4a^2 \int_{\sqrt{k}(a-\theta)}^{\infty}\dd z~\phi(z;0,1) + \int_{-\infty}^{\infty}\dd z~(z-\theta)^2\phi(z;\theta,1/k) \\
\leq {}& 8a^2\int_{\sqrt{k}(a-b)}^{\infty} \dd z~\frac{z}{\sqrt{k}(a-b)}\phi(z;0,1)
+\frac{1}{k}\\
= {}& \frac{8a^2}{\sqrt{k}(a-b)}\exp\bigg(-\frac{k(a-b)^2}{2}\bigg)+\frac{1}{k}.
\end{align*}
Consequently, we verify that condition A2 with $q=2$ is fulfilled. Next, we verify that condition A3 holds. We have
\begin{align*}
\sup_{\theta\in K}\{\log p(z|\hat{\theta}(z))-\log p(z|\theta)\}
\leq \sup_{\theta\in K} \frac{(z-\theta)^2}{2}
\leq \frac{(z-b)^2}{2}+\frac{(z+b)^2}{2}=z^2+b^2.
\end{align*}
Since moments of all orders exist and they are continuous in $\theta$, condition A3 is satisfied with $r=2$. 
	
Conditions A4 and A5 are also fulfilled, and the proofs are described in Appendix.
	
\begin{lemma}\label{lemma_restricted_normal_condA4}
For this model, condition A4 is satisfied.
\end{lemma}
\begin{proof}
See Appendix.
\end{proof}
	
\begin{lemma}\label{lemma_restricted_normal_condA5}
For this model, condition A5 is satisfied.
\end{lemma}
\begin{proof}
See Appendix. 
\end{proof}
	
In summary, the one-dimensional normal distributions with restricted mean satisfy conditions A1--A5.
\end{example}

\begin{remark}\label{remark_minimax_regret}
	\upshape
	As we will see later, numerous statistical models, including normal and Weibull distributions, satisfy a stronger condition than A5, i.e., the normalizing constant of CNML3 does not depend on the value of observations $x^N$ (see condition B2 and Theorem \ref{non-asymptotic_evaluation_of_projection_dist_CNML3}). In Example \ref{restricted_normal}, we verify that the one-dimensional normal model with restricted mean satisfies condition A5. However, this model does not satisfy the stronger condition (condition B2) and the normalizing constant of CNML3 does depend on $x^N$. 
	
	The quantity
	\[
	\log \bigg(\int_{X^M} \dd\mu (y^M) p(y^M|\hat{\theta}(x^N,y^M))\bigg)
	\]
	is not only the logarithm of the normalizing constant of CNML3 but also the minimax conditional regret-3 when we observe $x^N$ and predict $M$ future variables. Intuitively speaking, if the statistical model has ``uniformity'' such as group structure (for example location-scale models), the conditional regret-3 is equal irrespective of the observations. Even when the uniformity is not equipped with the model such as Example \ref{restricted_normal}, condition A5 is considered to hold because the information of future variables $y^M$ increases as $M$ goes to infinity and therefore the effect of $x^N$ on the conditional regret decreases.
\end{remark}

\begin{example}[Normal Distributions with Unknown Means]\label{normal_dist}
\upshape
The third example is the normal distribution with unknown means. Let $X=\mathbf{R}^d$ and $\Theta=\mathbf{R}^d$. We take a compact subset $K$ of $\Theta$ and fix a positive number $\delta>0$.

We consider a normal distribution with mean $\theta=(\theta^{(1)},\ldots,\theta^{(d)})\in\Theta$ and covariance matrix $\sigma^2 I_{d}$. Here, $\sigma^2>0$ is a known parameter, and $I_d$ is the $d\times d$ identity matrix. The probability density function is
\[
p(z|\theta)=\frac{1}{(2\pi\sigma^2)^{\frac{d}{2}}}\exp\bigg(-\frac{\sum_{i=1}^d(z^{(i)}-\theta^{(i)})^2}{2\sigma^2}\bigg),\quad z=(z^{(1)},\ldots,z^{(d)})\in\mathbf{R}^d.
\]
For any compact set $\tilde{K}\subset \mathbf{R}^d$, there exist $\theta_{\mathrm{min},\tilde{K}}^{(i)}:=\mathrm {min}_{\theta\in \tilde{K}} \theta^{(i)}$ and $\theta_{\mathrm{max},\tilde{K}}^{(i)}:=\mathrm{max}_{\theta\in \tilde{K}}, \theta^{(i)}$. For $\theta_1,\theta_2\in K_{\delta}$,  the log-likelihood function satisfies
\[
|\log p(z|\theta_1)-\log p(z|\theta_2)|\leq \sum_{i=1}^d \bigg(\frac{1}{\sigma^2}|z^{(i)}|+\frac{|\theta_{\mathrm{max},K_{\delta}}^{(i)}|+|\theta_{\mathrm{min},K_{\delta} }^{(i)} | }{2\sigma^2}\bigg)|\theta_1-\theta_2|.
\]
Therefore, condition A1 is satisfied with $p=2$. The MLE is the sample mean and its variance is $\E_{\theta}[|\hat{\theta}(z^k)-\theta|^2]=d\sigma^2/k$. Thus, condition A2 is satisfied with $q=2$. We have
\begin{align*}
\sup_{\theta\in K}\{\log p(z|\hat{\theta}(z))-\log p(z|\theta)\}
&= \sup_{\theta\in K} \sum_{i=1}^d \frac{(z^{(i)}-\theta^{(i)})^2}{2\sigma^2}\\
&\leq \sum_{i=1}^d \bigg\{\frac{(z^{(i)}-\theta_{\mathrm{min},K}^{(i)})^2}{2\sigma^2}+\frac{(z^{(i)}-\theta_{\mathrm {max},K}^{(i)})^2}{2\sigma^2}\bigg\}.
\end{align*}
Because moments of all orders exist and are continuous in $\theta$, condition A3 is satisfied with $r=2$.  
	
Since for any $\theta\in\Theta$ and for all $j=1,\ldots,d$,
\[
\E_{\theta} \bigg\{-\frac{1}{2\sigma^2}\sum_{i=1}^k \bigg(z_i^{(j)}-\frac{1}{k}\sum_{l=1}^k z_l^{(j)}\bigg)^2+\frac{1}{2\sigma^2}\sum_{i=1}^k (z_i^{(j)}-\theta)^2\bigg\}=\frac{1}{2},
\]
condition A4 is satisfied.
	
Finally, we show that condition A5 holds. Let $\overline{x^{(i)}}:=\sum_{j=1}^Nx_j^{(i)}/N$ and $\overline{y^{(i)}}:=\sum_{j=1}^My_j^{(i)}/M$. By the translation invariance of the Lebesgue measure,
\begin{align*}
\int_{\mathbf{R}^{dM}} & \dd y^M p(y^M|\hat{\theta}(x^N,y^M))\\
={}&\int_{\mathbf{R}^{dM}} \dd y^M \frac{1}{(2\pi\sigma^2)^{\frac{dM}{2}}}\exp\bigg(-\frac{1}{2\sigma^2}\sum_{i=1}^d\sum_{j=1}^M\bigg(y_j^{(i)}-\frac{N\overline{x^{(i)}}+M\overline{y^{(i)}}}{N+M}\bigg)^2\bigg)\\
={}&\int_{\mathbf{R}^{dM}} \dd z^M \frac{1}{(2\pi\sigma^2)^{\frac{dM}{2}}}\exp\bigg(-\frac{1}{2\sigma^2}\sum_{i=1}^d \sum_{j=1}^M\bigg(z_j^{(i)}-\frac{\sum_{k=1}^Mz_k^{(i)}}{N+M}\bigg)^2\bigg),
\end{align*}
where $z_j^{(i)}:=y_j^{(i)}-\sum_{k=1}^Nx_k^{(i)}/N$. Therefore the normalizing constant of CNML3 does not depend on $x^N$, and thus condition A5 is satisfied. In summary, the normal distributions satisfy conditions A1--A5.
\end{example}

\begin{example}[Exponential Distributions]\label{exp_dist}
\upshape
The fourth example is the exponential distribution. Let $X=(0,\infty)$ and $\Theta=(0,\infty)$. We take a compact set $K$ that is contained in $\Theta$. We fix a positive constant $\delta$ such that $\inf_{\theta\in K_{\delta}} \theta >0$. We define $\theta_{\mathrm{min}, K}:=\mathrm{min}_{\theta\in K}\theta > 0$, $\theta_{\mathrm{max}, K}:=\mathrm{max}_{\theta\in K}\theta <\infty$ and $\theta_{\mathrm{min}, K_{\delta}}:=\mathrm{min}_{\theta\in K_{\delta}}\theta > 0$.
	
The probability density function is 
\[
p(z|\theta) = \theta \exp(-\theta z),\quad z\in X,\quad \theta\in\Theta,
\]
and by the mean value theorem, for all $\theta_1,\theta_2\in K_{\delta}$,
\[
|\log p(z|\theta_1)-\log p(z|\theta_2)|\leq \bigg(\frac{1}{\theta_{\mathrm{min},K_{\delta}}}+z\bigg)|\theta_1-\theta_2|.
\]
Therefore, condition A1 with $p=2$ is satisfied. Condition A3 with $r=2$ is also satisfied because
\[
\\
\sup_{\theta\in K}\{\log p(z|\hat{\theta}(z))-\log p(z|\theta)\}\leq -\log z+|\log \theta_{\mathrm{min},K}|+|\log \theta_{\mathrm{max},K}|+|\theta_{\mathrm{max},K}|z,
\]
and
\[
\sup_{\theta\in K}\E_{\theta}[z^2] <\infty,\quad \sup_{\theta\in K}\E_{\theta}[(\log z)^2]<\infty.
\]
The MLE is $\hat{\theta}(z^k)=k/\sum_{i=1}^k z_i$ and $\sum_{i=1}^k z_i$ follows the gamma distribution with mean $k/\theta$ and variance $k/\theta^2$. Therefore,
\[
\E_{\theta}[|\hat{\theta}(z^k)-\theta|^2]=\theta^k \int_0^{\infty} \dd u \bigg(\frac{k}{u}-\theta\bigg)^2\frac{u^{k-1}e^{-\theta u}}{\Gamma (k)}=\frac{2(k+1)}{(k-1)(k-2)}\theta^2,
\]
and condition A2 is satisfied with $q=2$ because $0<\theta_{\mathrm{min},K}\leq \theta\leq \theta_{\mathrm{max},K}<\infty$ for all $\theta\in K$. Next, we verify that condition A4 holds.
\begin{align*}
\int_{X^k} \dd\mu(z^k) p(z^k|\theta) \log \frac{p(z^k|\hat{\theta}(z^k))}{p(z^k|\theta)}
&=k\log k-k-k\log \theta-k\E_{\theta}\bigg[\log \sum_{i=1}^k z_i\bigg]+\theta\E_{\theta}\bigg[\sum_{i=1}^k z_i\bigg]\\
&=k\log k-k\log \theta - k\int_0^{\infty} \dd u \log u \frac{\theta^k u^{k-1}}{\Gamma (k)}e^{-\theta u}\\
&=k\log k-k\log \theta -k(\psi(k)-\log\theta)\\
&=k(\log k - \psi(k)),
\end{align*}
where $\psi$ is the digamma function \citep{GradshteynRyzhik2007}. The digamma function is represented as
\[
\psi(k)=\log k -\frac{1}{2k} - 2\int_0^{\infty} \dd u\frac{u}{(u^2+k^2)(\exp(2\pi u)-1)}.
\]
Since $k^2\leq u^2+k^2$, 
\[
0\leq \int_0^{\infty} \dd u\frac{u}{(u^2+k^2)(\exp(2\pi u)-1)}\leq \frac{1}{k^2}\int_0^{\infty} \dd u\frac{u}{\exp(2\pi u)-1}=\frac{1}{24k^2}.
\]
Therefore,
\[
\lim_{k\to\infty}k(\log k - \psi(k))=\frac{1}{2},
\]
and thus, condition A4 is satisfied. Finally, we show that condition A5 holds. Let $\bar{x}=\sum_{i=1}^N x_i/N$ and $\bar{y}=\sum_{i=1}^M y_i/M$. The normalizing constant of CNML3 is
\begin{align*}
\int_{X^M} \dd\mu(y^M) p(y^M|\hat{\theta}(x^N,y^M))=\int_{X^M} \dd y^M \bigg(\frac{N+M}{N\bar{x}+M\bar{y}}\bigg)^M\exp\bigg(-\frac{N+M}{N\bar{x}+M\bar{y}}\sum_{i=1}^M y_i\bigg).
\end{align*}
Let $z_i=My_i/(N\bar{x})$ and $\bar{z}=\sum_{i=1}^Mz_i/M$. Then,
\[
\int_{X^M} \dd\mu(y^M) p(y^M|\hat{\theta}(x^N,y^M))=\int_{X^M} \dd z^M \bigg(\frac{N+M}{M+M\bar{z}}\bigg)^M\exp\bigg(-\frac{N+M}{M+M\bar{z}}\sum_{i=1}^M z_i\bigg).
\]
This is independent of $x^N$. In conclusion, the exponential distributions satisfy conditions A1--A5.
\end{example}

Thus far, we have considered asymptotic situations, but next, we provide a non-asymptotic result. We state conditions for the result.

\begin{description}
\item[B1.] For all $\theta\in K$, and for all $N$ and $M$,
\[
\int_{X^N\times X^M}\dd\mu(x^N, y^M)p(x^N,y^M|\theta) \log\frac{p(y^M|\theta)}{p(y^M|\hat{\theta}(x^N,y^M))}
\]
does not depend on $\theta$.

\item[B2.] For all $\theta\in K$, and for all $N$ and $M$,
\[
\log \bigg(\int_{X^M} \dd\mu(y^M) p(y^M|\hat{\theta}(x^N,y^M))\bigg)
\]
does not depend on $x^N$.
\end{description}

\begin{theorem}\label{non-asymptotic_evaluation_of_projection_dist_CNML3}
Let $K$ be a compact set that is contained in the interior of  $\Theta$ and assume that $p(z|\theta)$ is strictly positive for all $z\in X$ and $\theta\in K$. Assume also that conditions B1 and B2 are satisfied. 
	
Then, for any $\pi\in\mathcal{P}_K$ and for all $N$ and $M$,
\begin{align}\label{non-asymptotic_formula_of_projection_dist_CNML3}
D_{K,q_{\mathrm{CNML3}}}^{N,M}(\pi)+R_{\mathrm{KL}}^{N,M}(\pi, p_{\pi})=C_*^{N,M},
\end{align}
where $C_*^{N,M}$ is a constant that is independent of $\pi$. Therefore, BPCNML3 exactly coincides with the Bayesian predictive density based on the LIP.
\end{theorem}

\begin{proof}
The left-hand side of \eqref{non-asymptotic_formula_of_projection_dist_CNML3} is
\begin{align*}
D_{K,q_{\mathrm{CNML3}}}^{N,M}&(\pi)+R_{\mathrm{KL}}^{N,M}(\pi, p_{\pi})\\
={}&\int_{\Theta}\pi(\dd\theta)\bigg\{\int_{X^N\times X^M}\dd\mu(x^N, y^M)
p(x^N,y^M|\theta) \log\frac{p(y^M|\theta)}{p(y^M|\hat{\theta}(x^N,y^M))}\bigg\}\\
&+\int_{\Theta}\pi(\dd\theta)\bigg\{ \int_{X^N} \dd\mu(x^N) p(x^N|\theta) \log \bigg(\int_{X^M} \dd\mu(y^M) p(y^M|\hat{\theta}(x^N,y^M))\bigg)\bigg\}.
\end{align*}
By assumptions B1 and B2, the claim is verified.
\end{proof}

\begin{example}[One-Dimensional Normal Distribution with Unknown Mean]
\upshape
In Example \ref{normal_dist}, we show that the normal distribution satisfies condition B2. Here, we verify that condition B1 holds. Assume that $x^N$ and $y^M$ are independent and identically normally distributed with unknown mean $\theta$ and variance $1$. Let $\bar{x}:=\sum_{i=1}^Nx_i/N$ and let $\bar{y}:=\sum_{i=1}^M y_i/M$. Then,
\begin{align*}
\int_{X^N\times X^M}&\dd\mu(x^N, y^M)p(x^N,y^M|\theta) \log\frac{p(y^M|\theta)}{p(y^M|\hat{\theta}(x^N,y^M))}\\
={}&\E_{\theta}
\sum_{i=1}^M\frac{1}{2}\bigg\{-(y_i-\theta)^2+\bigg(y_i-\frac{N\bar{x}+M\bar{y}}{N+M}\bigg)^2\bigg\}\\
={}&-\frac{M}{2}+\frac{M(1+\theta^2)}{2}-\frac{NM\theta^2+M(1+M\theta^2)}{N+M}+\frac{M\theta^2}{2}+\frac{M}{2(N+M)}\\
={}&-\frac{M}{2(N+M)}.
\end{align*}
Condition B1 is satisfied, and thus, Theorem \ref{non-asymptotic_evaluation_of_projection_dist_CNML3} holds.
\end{example}

\begin{example}[Weibull Distribution with Unknown Scale Parameter]\label{Weibull_dist}
\upshape
Let $X=(0,\infty)$ and $\Theta=(0,\infty)$. We consider the Weibull distribution with unknown scale parameter $\theta\in\Theta$ and known shape parameter $k\in (0,\infty)$. The Weibull distributions are widely known to include numerous other probability distributions, such as the exponential distributions ($k=1$) and the Rayleigh distributions ($k=2$).
	
The probability density function is
\[
p(z|\theta)=\frac{k}{\theta}\bigg(\frac{z}{\theta}\bigg)^{k-1}\exp\bigg\{-\bigg(\frac{x}{\theta}\bigg)^k\bigg\},\quad z\in X.
\]
The MLE is 
\[
\hat{\theta}(z^n)=\bigg(\frac{1}{n}\sum_{i=1}^n z_i^k \bigg)^{\frac{1}{k}}.
\]
First we show that condition B1 is satisfied. We have
\begin{align*}
\int_{X^N\times X^M}&\dd\mu(x^N, y^M)p(x^N,y^M|\theta) \log\frac{p(y^M|\theta)}{p(y^M|\hat{\theta}(x^N,y^M))}\\
={}&\E_{\theta}\sum_{i=1}^M\bigg\{-k\log \theta
+k\log \hat{\theta}(x^N,y^M)-\frac{y_i^k}{\theta^k}+\frac{y_i^k}{(\hat{\theta}(x^N,y^M))^k}\bigg\}\\
={}&\E_{\theta}\bigg\{M\log \Bigg(\sum_{i=1}^N\frac{x_i^k}{\theta^k}
+\sum_{i=1}^M\frac{y_i^k}{\theta^k}\bigg)-\sum_{i=1}^M\frac{y_i^k}{\theta^k}+\frac{(N+M)\sum_{i=1}^M \frac{y_i^k}{\theta^k}}
{\sum_{i=1}^N\frac{x_i^k}{\theta^k}+\sum_{i=1}^M\frac{y_i^k}{\theta^k}}\Bigg\}-M\log (N+M).
\end{align*}
If a random variable $Z$ follows the Weibull distribution with scale parameter $\theta$ and shape parameter $k$, then $(Z/\theta)^k$ follows the exponential distribution with mean $1$. In addition, if two random variables $Z_1$ and $Z_2$ follow the gamma distributions with common scale parameter $\xi$ and shape parameters $\alpha$ and $\beta$, respectively, then $Z_1/(Z_1+Z_2)$ follows the beta distribution with shape parameters $\alpha$ and $\beta$. From these facts and the reproductive property of the gamma distribution,
\begin{align*}
\int_{X^N\times X^M} & \dd\mu(x^N, y^M)p(x^N,y^M|\theta) \log\frac{p(y^M|\theta)}{p(y^M|\hat{\theta}(x^N,y^M))}\\
={}&M\psi(N+M)-M+(N+M)\times \frac{M}{N+M}-M\log (N+M)\\
={}&M(\psi(N+M)-\log (N+M)),
\end{align*}
where $\psi$ is the digamma function. Hence, condition B1 is fulfilled. Because we can verify that condition B2 holds in the same manner as Example \ref{exp_dist}, we omit the proof.
\end{example}

\section{Numerical Experiments}\label{numerical_experiments}
In Example \ref{multinomial_dist}, we verify that the multinomial distribution satisfies condition A1--A5 and thus, Theorem \ref{asymptotic_dist_and_KLrisk_from_cnml3_to_Bayes} holds. In this section, we confirm the validity of Theorem \ref{asymptotic_dist_and_KLrisk_from_cnml3_to_Bayes} for the binomial distribution through numerical experiments. 

We explain the settings of the numerical experiments. Let $\Theta=[0,1]$ and $K=[0.1,0.9]$. Since $\mathcal{P}_K$ is infinite-dimensional space, we approximate $\mathcal{P}_K$ by the set of discrete distributions $\tilde{P}_K^{100}$:
\[
\tilde{P}_K^{100}:=\bigg\{\sum_{i=0}^{100} \pi_i \delta_{0.1+0.08i}(\dd\theta) \bigg| 0\leq \pi_i \leq 1~\mathrm{ for}~\mathrm{all}~i,~\sum_{i=0}^{100}\pi_i=1.\bigg\},
\]
where $\delta_{a}(\dd\theta)$ denotes the Dirac measure with support $a\in\Theta$. By numerical optimization, we calculate the approximation of the LIP
\begin{align*}
\tilde{\pi}_{K,\mathrm{LIP}}^{N,M}
:={}&\argmax_{\pi\in\tilde{\mathcal{P}}_K}R_{\mathrm{KL}}^{N,M}(\pi,p_{\pi})\\
={}&\argmax_{\pi\in\tilde{\mathcal{P}}_K} \sum_{i,j,k}\pi_i
\begin{pmatrix}
N\\
j
\end{pmatrix}
\begin{pmatrix}
M\\
k
\end{pmatrix}
\theta_i^{j+k}(1-\theta_i)^{N+M-j-k} \log \frac{\theta_i^k(1-\theta_i)^{M-k}p_{\pi}(j)}{p_{\pi}(j,k)},
\end{align*}
and BPCNML3
\begin{align*}
\tilde{\pi}_{K,q_{\mathrm{CNML3}}}^{N,M}
:={}&\argmin_{\pi\in\tilde{\mathcal{P}}_K}D_{K,q_{\mathrm{CNML3}}}^{N,M}(\pi)\\
={}&\argmin_{\pi\in\tilde{\mathcal{P}}_K} \sum_{i,j,k}\pi_i
\begin{pmatrix}
N\\
j
\end{pmatrix}
\begin{pmatrix}
M\\
k
\end{pmatrix}
\theta_i^{j+k}(1-\theta_i)^{N+M-j-k} \log \frac{p_{\pi}(j,k)}{(\hat{\theta}_{j,k})^k(1-\hat{\theta}_{j,k})^{M-k}p_{\pi}(j)},
\end{align*}
where $\theta_i:=0.1+0.08i$, $\hat{\theta}_{j,k}:=(j+k)/(N+M)$, $p_{\pi}(j):=\sum_{i=0}^{100}\pi_i \theta_i^j(1-\theta_i)^{N-j}$ and $p_{\pi}(j,k):=\sum_{i=0}^{100}\pi_i \theta_i^{j+k}(1-\theta_i)^{N+M-j-k}$. We used the free software R \citep{R2009} and constrOptim function for the optimization.

\begin{figure}[tbp]
\centering
\begin{minipage}{0.49\hsize}
	\includegraphics[width=7.5cm]{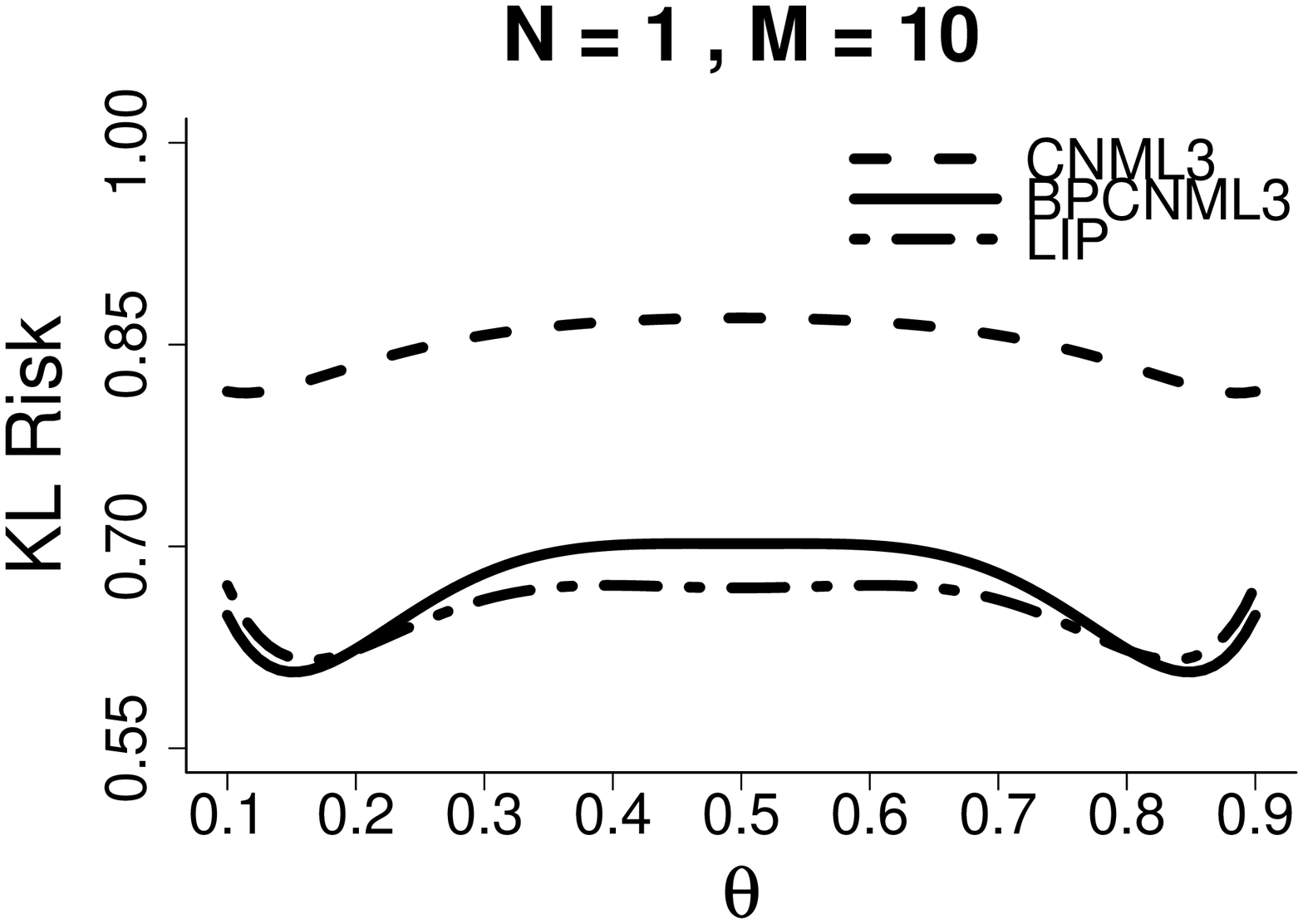}
\end{minipage}
\begin{minipage}{0.49\hsize}
	\includegraphics[width=7.5cm]{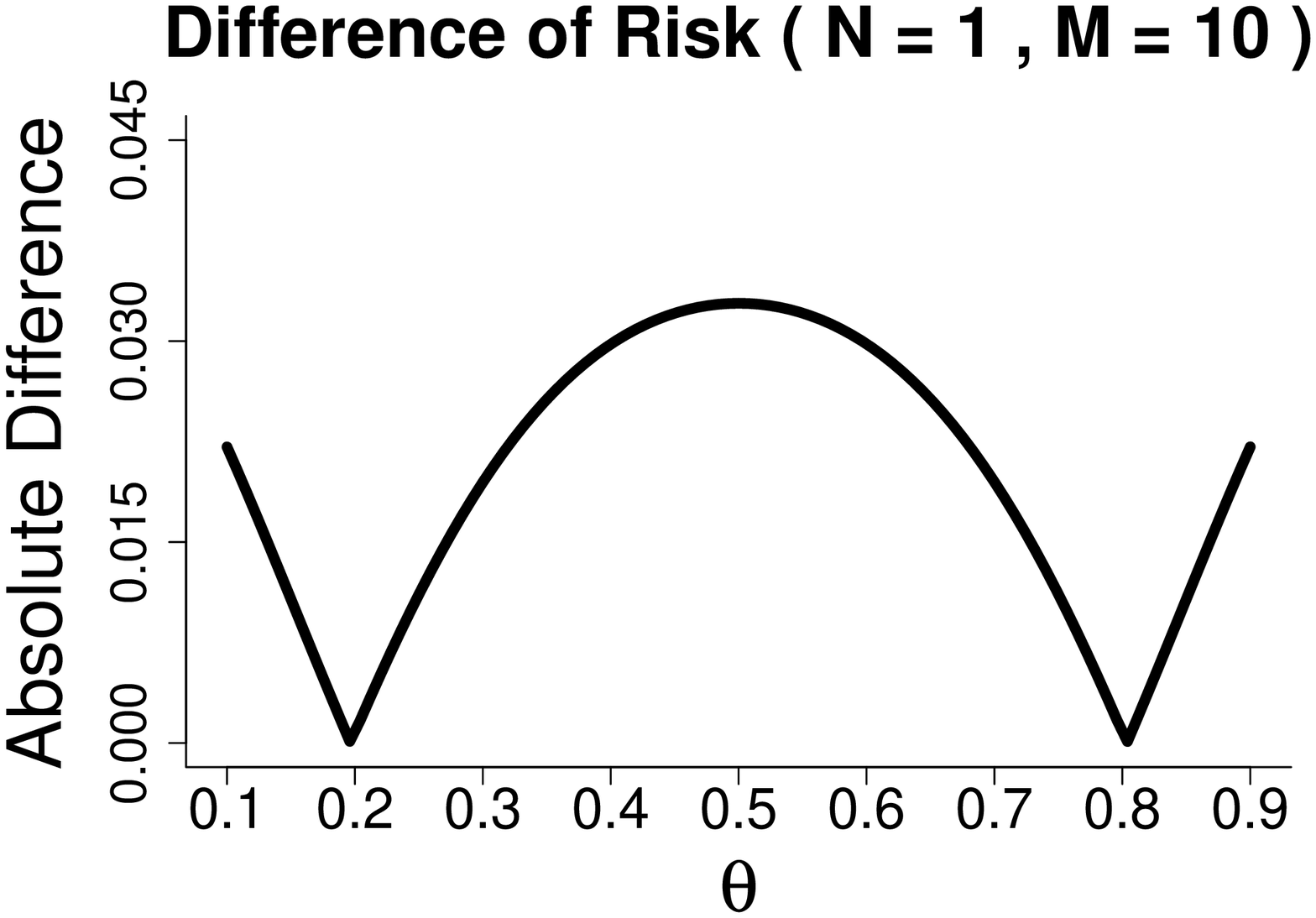}
\end{minipage}
\caption{Comparison of KL risk when $N=1$, $M=10$. The right panel shows the absolute difference of  KL risk between BPCNML3 and BPDLIP.}
\label{Binomial_N1M10}

\begin{minipage}{0.49\hsize}
	\includegraphics[width=7.5cm]{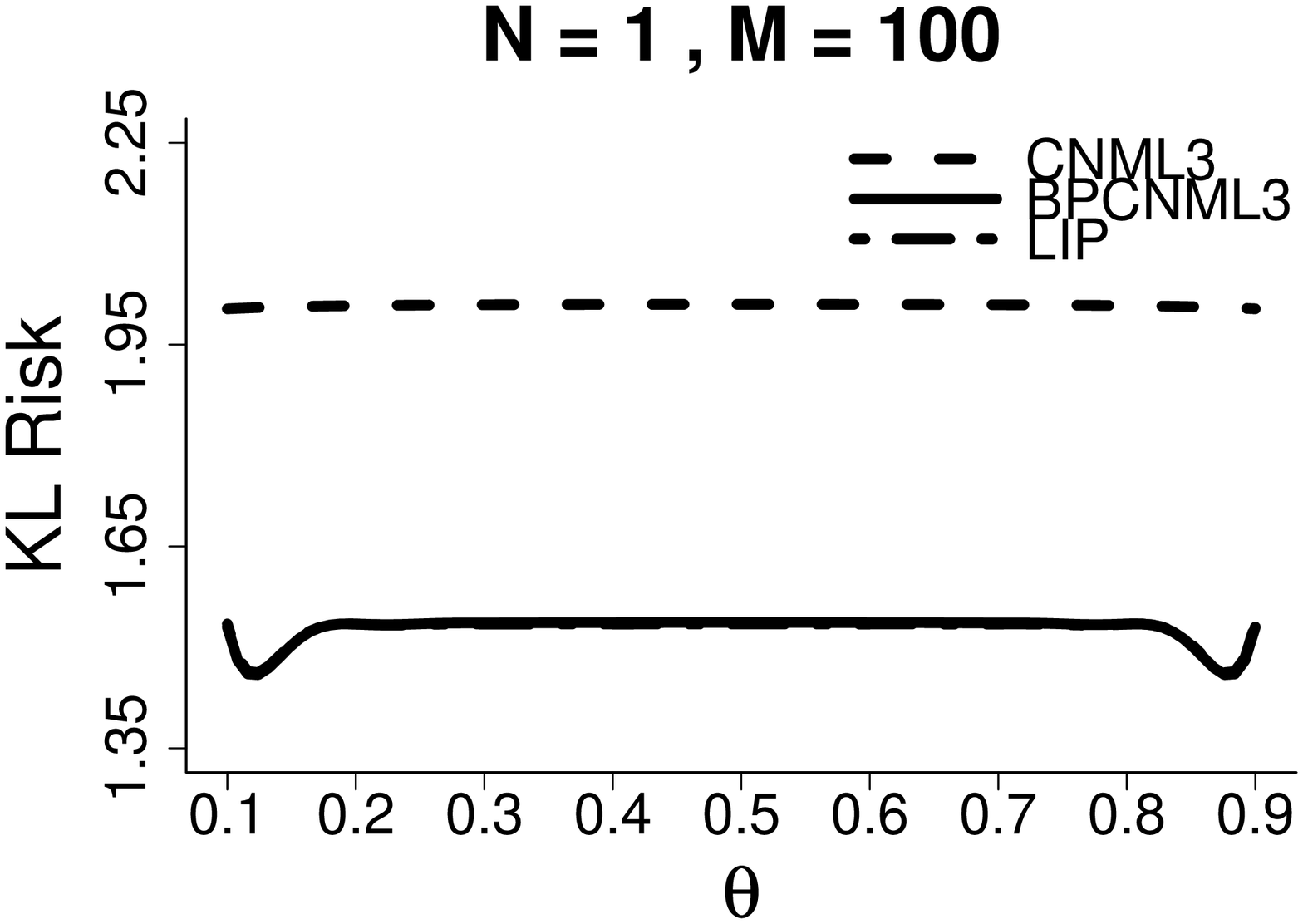}
\end{minipage}
\begin{minipage}{0.49\hsize}
	\includegraphics[width=7.5cm]{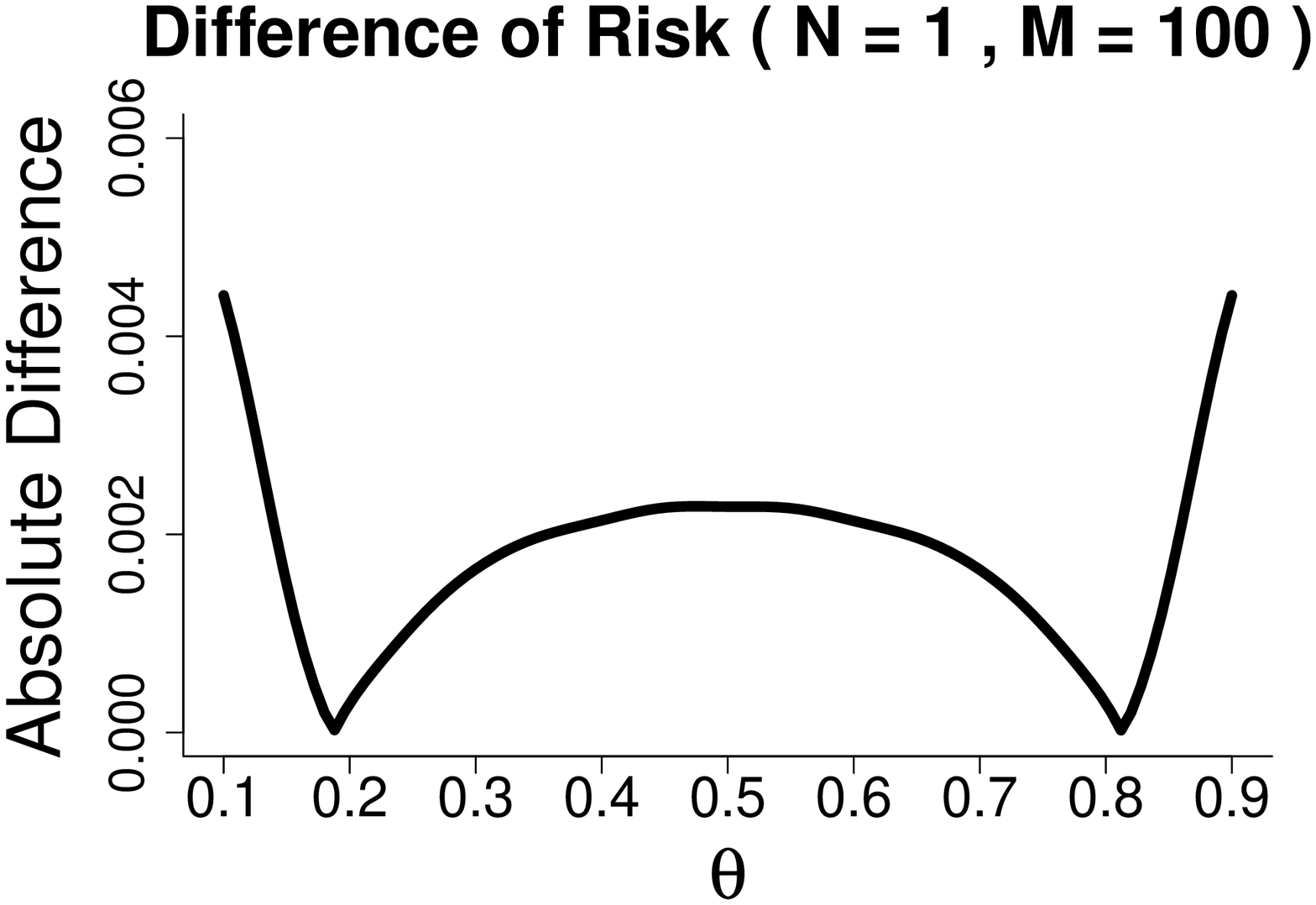}
\end{minipage}
\caption{Comparison of KL risk when $N=1$, $M=100$.  Since the KL risk of BPCNML3 is almost the same as that of BPDLIP, we plot the absolute difference of KL risk between BPCNML3 and BPDLIP in the right panel.}
\label{Binomial_N1M100}

\begin{minipage}{0.49\hsize}
	\includegraphics[width=7.5cm]{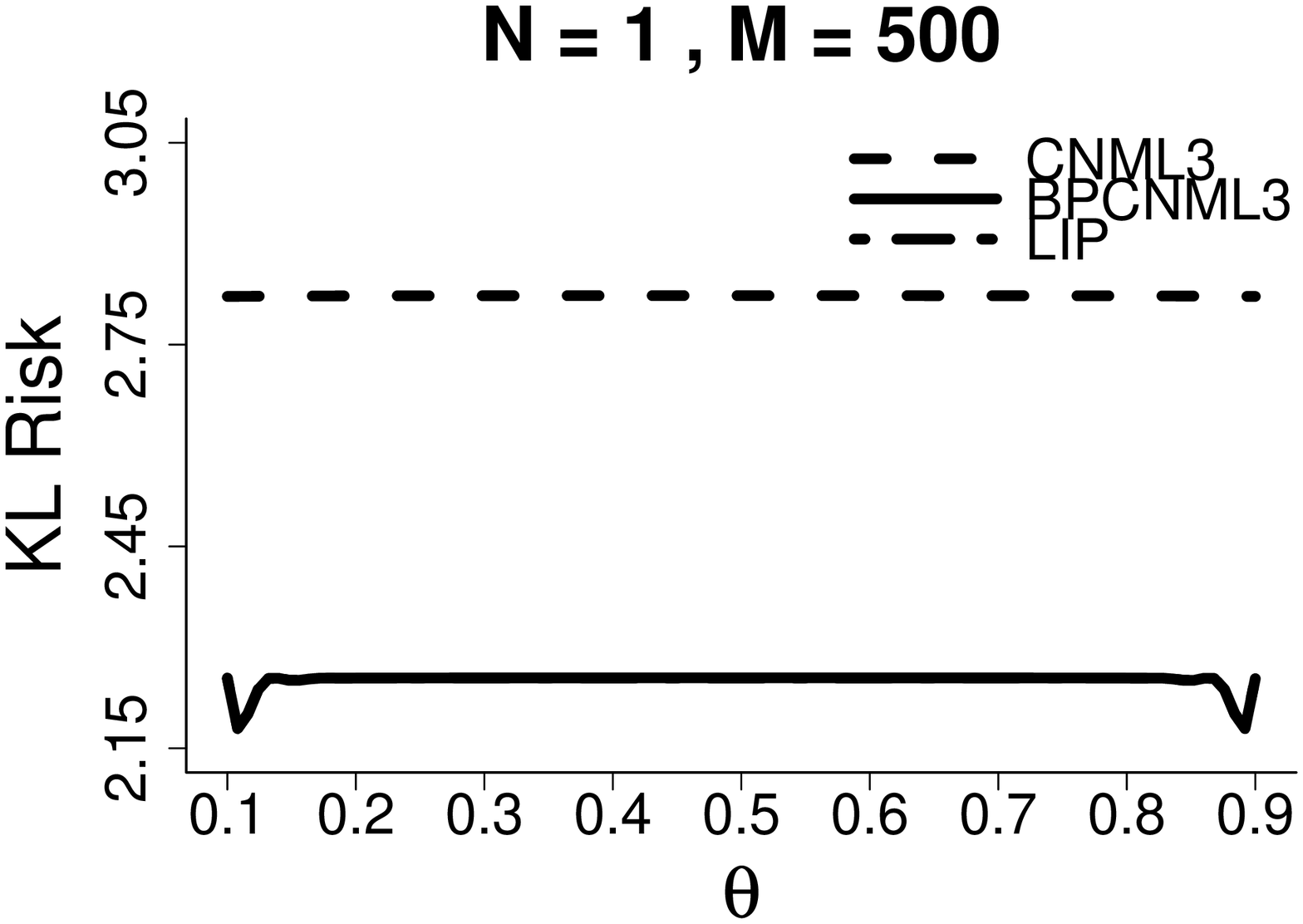}
\end{minipage}
\begin{minipage}{0.49\hsize}
	\includegraphics[width=7.5cm]{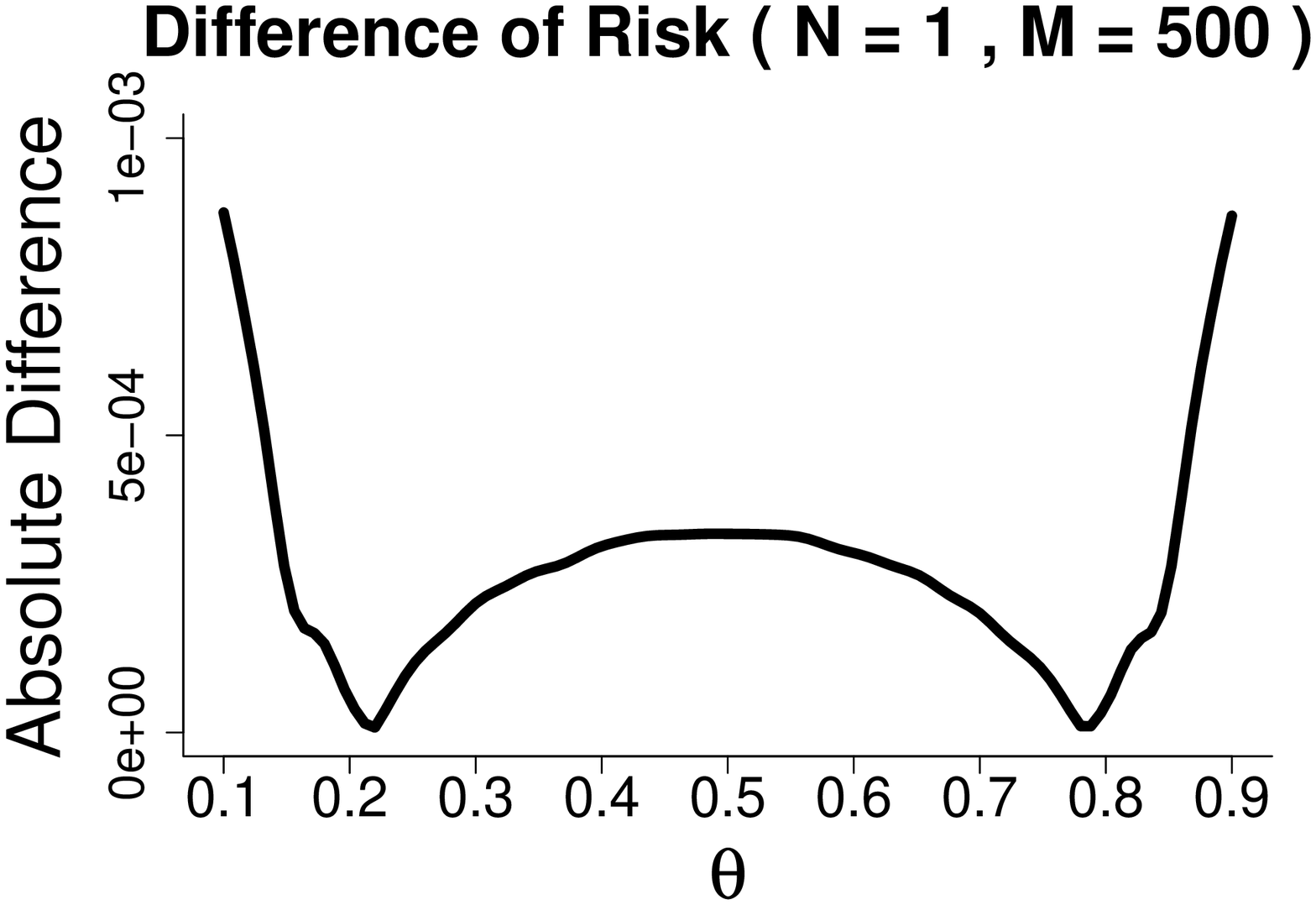}
\end{minipage}
\caption{Comparison of KL risk when $N=1$, $M=500$. Since the KL risk of BPCNML3 is almost the same as that of BPDLIP, we plot the absolute difference of KL risk between BPCNML3 and BPDLIP in the right panel.}
\label{Binomial_N1M500}
\end{figure}

Figure \ref{Binomial_N1M10}--\ref{Binomial_N1M500} show the result of comparison of KL risk among CNML3, BPCNML3, and Bayesian predictive densities based on LIP (simply abbreviated to BPDLIP) when $N=1$ and $M=10,100,500$. When $N=1$ and $M=100, 500$ (Figure \ref{Binomial_N1M100} and \ref{Binomial_N1M500}), KL risk of BPCNML3 is almost the same as that of BPDLIP. Therefore, we plot the absolute difference of KL risk between BPCNML3 and BPDLIP.

Implications from the figures are twofold. First, KL risk of BPCNML3 is much lower than that of CNML3. Notably, for submodels of the multinomial distributions, \cite{Komaki2011} showed that KL risk of the Bayes projection of predictive density $q$ is not larger than that of $q$. In addition, the amount of reduction increases as $M$ increases.
Second, we find that the difference of KL risk between BPCNML3 and BPDLIP goes to zero as $M$ increases. This finding implies that BPCNML3 is asymptotically identical to BPDLIP.

\section{Conclusion}\label{conclusion}
In this study, we discussed the relations between the Bayes projection of CNML3 (BPCNML3) and the Bayesian predictive density based on the LIP (BPDLIP). In Theorem \ref{asymptotic_evaluation_of_projection_dist_CNML3}, we proved that the sum of the Bayes projection divergence of CNML3 and the conditional mutual information is asymptotically constant. Roughly speaking, this result implies that the BPCNML3 is asymptotically identical to the BPDLIP. The numerical results in Section \ref{numerical_experiments} confirmed that the BPCNML3 is asymptotically identical to the BPDLIP for the binomial model. Under stronger conditions B1 and B2, we showed that the BPCNML3 exactly coincides with the BPDLIP in Theorem \ref{non-asymptotic_evaluation_of_projection_dist_CNML3}.

Our results shed light on the connection between CNML3 and LIPs. Although CNML2 has received the most attention among CNML distributions, we argue that CNML3, not CNML2, is more in line with the minimax KL risk approach and is the most important predictive density among CNML distributions.

Finally, we provide our future plans for this study. The plans are threefold. First, we will study the sufficient conditions for A5 and B2. These conditions are concerned with the conditional minimax regret-3. As reported in Remark \ref{remark_minimax_regret}, we believe that numerous regular statistical models satisfy these conditions. Second, we will address the boundary of the parameter space. In the same manner as \cite{ClarkeBarron1994}, we restricted the support set of the prior distributions that should be contained in the fixed compact set. Using the methods such as in \cite{XieBarron2000} or \cite{Komaki2012}, we may treat the boundary of the parameter space. Finally, we plan to study the predictive performance of the BPDLIP under the conditional regret-3. It is an interesting study because it parallels to the study of \cite{XieBarron2000}.

\appendix
\section{Proofs of Lemmas}
\subsection{Proof of Lemma 2}
\begin{proof}
We define several notations as follows:
\begin{align*}
I_{N,M}(\theta)&:=\int_{X^N\times X^M} \dd\mu(x^N,y^M)~ p(x^N, y^M|\theta) \log \frac{p(y^M|\theta)}{p(y^M|\hat{\theta}(x^N, y^M))}+\frac{d}{2},\\
R_{k}(\theta)&:=\int_{X^k} \dd\mu(z^k)~ p(z^k|\theta) \log \frac{p(z^k|\theta)}{p(z^k|\hat{\theta}(z^k))}+\frac{d}{2}.
\end{align*}
Note that since $p(x^N,y^M|\hat{\theta}(x^N,y^M))=p(x^N|\hat{\theta}(x^N,y^M))p(y^M|\hat{\theta}(x^N,y^M))>0$ for all $x^N$ and $y^M$,
\[
p(x^N|\hat{\theta}(x^N,y^M))>0,\quad p(y^M|\hat{\theta}(x^N,y^M))>0.
\]
Since $p(y^M|\hat{\theta}(y^M))\geq p(y^M|\hat{\theta}(x^N,y^M))$
\begin{align}\label{lower_bound}
R_{M}(\theta)\leq I_{N,M}(\theta),\quad \forall{\theta}\in K.
\end{align}
For $\theta\in K$, the integrand in the claim of Lemma \ref{lemma_approx_bias_of_loglikelihood} is decomposed as
\begin{align}\label{decomposition_of_integrand}
\frac{p(y^M|\theta)}{p(y^M|\hat{\theta}(x^N,y^M))}=\frac{p(x^N,y^M|\theta)}{p(x^N,y^M|\hat{\theta}(x^N,y^M))}\frac{p(x^N|\hat{\theta}(x^N,y^M))}{p(x^N|\theta)}.
\end{align}
By condition A1, for $(x^N,y^M)\in \{\hat{\theta}(x^N,y^M)\in K_{\delta}\}$
\begin{align*}
\log \frac{p(x^N|\hat{\theta}(x^N,y^M))}{p(x^N|\theta)} \leq |\hat{\theta}(x^N,y^M)-\theta|\sum_{i=1}^N L_{K_{\delta}}(x_i).
\end{align*}
In addition, by condition A3, for $(x^N,y^M)\in \{\hat{\theta}(x^N,y^M)\not\in K_{\delta}\}$
\begin{align*}
\log \frac{p(x^N|\hat{\theta}(x^N,y^M))}{p(x^N|\theta)}&=\sum_{i=1}^N \{\log p(x_i|\hat{\theta}(x^N,y^M))-\log p(x_i|\theta)\}\\
&\leq \sum_{i=1}^N \{\log p(x_i|\hat{\theta}(x_i))-\log p(x_i|\theta)\}\\
&\leq  \sum_{i=1}^N T_K(x_i).
\end{align*}
By the H\"{o}lder inequality, for all $\theta \in K$
\begin{align*}
\int_{X^N\times X^M} & \dd\mu(x^N,y^M) p(x^N,y^M|\theta) \log \frac{p(x^N|\hat{\theta}(x^N,y^M))}{p(x^N|\theta)} \\
\leq {}& \sup_{\theta\in K} \int_{\{\hat{\theta}(x^N,y^M)\in K_{\delta}\}}\dd\mu(x^N,y^M) p(x^N,y^M|\theta) |\hat{\theta}(x^N,y^M)-\theta|\sum_{i=1}^N L_{K_{\delta}}(x_i)\\
&+ \sup_{\theta\in K}\int_{\{\hat{\theta}(x^N,y^M)\not\in K_{\delta}\}} \dd\mu(x^N,y^M) p(x^N,y^M|\theta) \sum_{i=1}^N T_K(x_i)\\
\leq {}& \sup_{\theta\in K} \bigg\{\int_{X^N} \dd\mu(x^N) p(x^N|\theta) \bigg(\sum_{i=1}^N L_{K_{\delta}}(x_i)\bigg)^p\bigg\}^{\frac{1}{p}}\\
&\times  \sup_{\theta\in K} \bigg\{\int_{X^N\times X^M} \dd\mu(x^N,y^M)p(x^N,y^M|\theta) |\hat{\theta}(x^N,y^M)-\theta|^q\bigg\}^{\frac{1}{q}}\\
&+ \sup_{\theta\in K}\bigg\{P^{N+M}_{\theta}\bigg(\hat{\theta}(x^N,y^M)\not\in K_{\delta}\bigg)\bigg\}^{\frac{1}{s}}
\sup_{\theta\in K}\bigg\{ \int_{X^N} \dd\mu(x^N) p(x^N|\theta) \bigg(\sum_{i=1}^N T_K(x_i)\bigg)^r \bigg\}^{\frac{1}{r}},
\end{align*}
where $s$ satisfies $1/r+1/s=1$.
We denote the upper bound by $U_{M,N}$. Note that $U_{M,N}$ is nonnegative and does not depend on $\theta$. By conditions A1--A3 and Lemma \ref{sufficient_conditions_for_uniform_consistency}, we have
\begin{align}\label{upperbound_of_reminder}
\lim_{M\to\infty} U_{M,N} =0.
\end{align}
From \eqref{lower_bound} and \eqref{decomposition_of_integrand},
\begin{align*}
R_{M}(\theta) \leq I_{M,N}(\theta)\leq R_{M+N}(\theta)+U_{M,N}.
\end{align*}
Therefore, since $|I_{M,N}(\theta)|\leq \max\{|R_{M}(\theta)|, |R_{M+N}(\theta)|+U_{M,N}\}$
\begin{align*}
\sup_{\theta\in K}|I_{M,N}(\theta)|\leq& \max\big\{\sup_{\theta\in K}|R_{M}(\theta)|,~ \sup_{\theta\in K}|R_{M+N}(\theta)|+U_{M,N}\big\}\\
\leq &\sup_{\theta\in K}|R_M(\theta)|+\sup_{\theta\in K}|R_{M+N}(\theta)| + U_{M,N}.
\end{align*}
By condition A4 and \eqref{upperbound_of_reminder}, we have
\[
\lim_{M\to\infty}\sup_{\theta\in K}|I_{M,N}(\theta)|=0. 
\]
\end{proof}

\subsection{Proof of Lemma 4}
\begin{proof}
We define a family of polynomials with one variable $t$ as follows:
\[
f_{M,a}^{(2)}(t):=\sum_{i=0}^M
\begin{pmatrix}
M\\
i\\
\end{pmatrix}
(t+i)^i(M+a-t-i)^{M-i},
\]
where $M$ is a positive integer and $a$ is a real number. We also define $f_{0,a}^{(2)}(t)\equiv 1$. If we set $a=N$ and $x\in\{0,1,\ldots, N\}$, then $f_{M,N}(x)/(M+N)^M$ is the normalizing constant of CNML3 for the binomial distributions with observations $x^N$ satisfying $\sum_{i=1}^N x_i=x$. For any nonnegative integer $M$ and any real number $a$, we first prove that $f_{M,a}^{(2)}$ does not depend on the value of $t$, i.e., $f_{M,a}^{(2)}$ is a constant function.
	
It suffices to show that for any real number $a$,
\begin{align}\label{claim_induction}
\frac{\dd}{\dd t}f_{M,a}^{(2)}(t)=0,\quad \forall{t}\in\mathbf{R},
\end{align}
since $f_{M,a}^{(2)}$ is a polynomial in $t$. We prove this by mathematical induction with respect to $M$. For $M=0$, \eqref{claim_induction} is evident by the definition of $f_{0,a}^{(2)}$. Assume that \eqref{claim_induction} holds for $M=m$ and any $a\in\mathbf{R}$. From this assumption, $f_{m,a}^{(2)}$ is a constant function. Then,
\begin{align*}
\frac{\dd}{\dd t}f_{m+1,a}^{(2)}(t)={}&\frac{\dd}{\dd t}\bigg\{(m+1+a-t)^{m+1}+(t+m+1)^{m+1}\bigg\}\\
&+\sum_{i=1}^m
\begin{pmatrix}
m+1\\
i
\end{pmatrix}
\frac{\dd}{\dd t}(t+i)^i(m+1+a-t-i)^{m+1-i}\\
={}&-(m+1)(m+1+a-t)^m+(m+1)(t+m+1)^m\\
&+\sum_{i=1}^m
\begin{pmatrix}
m+1\\
i
\end{pmatrix}
i(t+i)^{i-1}(m+1+a-t-i)^{m+1-i}\\
&-\sum_{i=1}^m
\begin{pmatrix}
m+1\\
i
\end{pmatrix}
(m+1-i)(t+i)^i(m+1+a-t-i)^{m-i}.
\end{align*}
Since
\[
\begin{pmatrix}
m+1\\
i
\end{pmatrix}
i=(m+1)
\begin{pmatrix}
m\\
i-1
\end{pmatrix}
,\quad 
\begin{pmatrix}
m+1\\
i
\end{pmatrix}
(m+1-i)=(m+1)
\begin{pmatrix}
m\\
i
\end{pmatrix}
,
\]
hold, 
	
\begin{align*}
\frac{\dd}{\dd t}f_{m+1,a}^{(2)}(t)
={}&-(m+1)(m+1+a-t)^m+(m+1)(t+m+1)^m\\
&+(m+1)\sum_{i=1}^m
\begin{pmatrix}
m\\
i-1
\end{pmatrix}
(t+i)^{i-1}(m+1+a-t-i)^{m+1-i}\\
&-(m+1)\sum_{i=1}^m
\begin{pmatrix}
m\\
i
\end{pmatrix}
(t+i)^i(m+1+a-t-i)^{m-i}\\
={}&(m+1)(t+m+1)^m\\
&+(m+1)\sum_{i=0}^{m-1}
\begin{pmatrix}
m\\
i
\end{pmatrix}
(t+i+1)^{i}(m+a-t-i)^{m-i}\\
&-(m+1)\sum_{i=0}^m
\begin{pmatrix}
m\\
i
\end{pmatrix}
(t+i)^i(m+1+a-t-i)^{m-i}\\
={}&(m+1)\sum_{i=0}^{m}
\begin{pmatrix}
m\\
i
\end{pmatrix}
(t+i+1)^{i}(m+a-t-i)^{m-i}\\
&-(m+1)\sum_{i=0}^m
\begin{pmatrix}
m\\
i
\end{pmatrix}
(t+i)^i(m+1+a-t-i)^{m-i}\\
={}&(m+1)\{f_{m,a+1}(t+1)-f_{m,a+1}(t)\}.
\end{align*}
By the assumption of the induction, $f_{m,a+1}(t+1)-f_{m,a+1}(t)=0$. Therefore, 
\[
\frac{\dd}{\dd t}f_{m+1,a}^{(2)}(t)=0, 
\]
and \eqref{claim_induction} is verified for any $M$ and $a$. In addition, from this result, the claim of Lemma \ref{lemma_multinomial_normconstants_of_CNML3} is verified for the binomial distributions.
	
Next, we show that the claim holds for $(d+1)$-nomial distributions $(d\geq 2)$. We define a family of polynomials with $d$ variables $(t_1,\ldots, t_d)$ as follows:
\[
f_{M,a}^{(d+1)}(t_1,\ldots, t_d):=\sum_{\substack{0\leq i_1,\ldots, i_d\leq M,\\ \sum_{l=1}^d i_l =M}}
\frac{M!\prod_{l=1}^d (t_l+i_l)^{i_l}}{i_1!\ldots i_k!(M-\sum_{l=1}^ki_l)!}
\bigg(M+a-\sum_{l=1}^d (t_l+i_l)\bigg)^{M-\sum_{l=1}^d i_l}.
\]
where $M$ is a positive integer and $a$ is a real number. For the same reason discussed in the case of the binomial distributions, it suffices to show that $f_{M,a}^{(d+1)}$ is a constant function to verify that the normalizing constant of CNML3 for $(d+1)$-nomial distribution is independent of $x^N$. 
	
Note that $f_{M,a}^{(d+1)}$ is symmetric with respect to any permutation of variables, i.e., for any permutation $\sigma$,
\[
f_{M,a}^{(d+1)}(t_{\sigma(1)},\ldots, t_{\sigma(d)})=f_{M,a}^{(d+1)}(t_1,\ldots, t_d).
\] 
Therefore, it is sufficient to show that
\[
\frac{\partial}{\partial t_1} f_{M,a}^{(d+1)}(t_1,\ldots, t_d)=0,
\]
since $f_{M,a}^{(d+1)}$ is a symmetric polynomial in $(t_1,\ldots, t_d)$.
\begin{align*}
\frac{\partial}{\partial t_1}& f_{M,a}^{(d+1)}(t_1,\ldots, t_d)\\
&= \frac{\partial}{\partial t_1} \sum_{\substack{0\leq i_2,\ldots, i_d\leq M,\\ \sum_{l=2}^d i_l \leq M}}
\sum_{i_1=0}^{M-\sum_{l=2}^d i_l}
\frac{M!\prod_{l=1}^d (t_l+i_l)^{i_l}}{i_1!\ldots i_d!(M-\sum_{l=1}^di_l)!}\bigg(M+a-\sum_{l=1}^d (t_l+i_l)\bigg)^{M-\sum_{l=1}^d i_l}\\
&=
\frac{\partial}{\partial t_1} \sum_{\substack{0\leq i_2,\ldots, i_d\leq M,\\ \sum_{l=2}^d i_l \leq M}}
\frac{M!\prod_{l=2}^d (t_l+i_l)^{i_l}}{i_2!\ldots i_d!(M-\sum_{l=2}^d i_l)!}
\\
&~~\times
\sum_{i_1=0}^{M-\sum_{l=2}^d i_l}
\frac{(M-\sum_{l=2}^d i_l)!(t_1+i_1)^{i_1}}{i_1!(M-\sum_{l=2}^d i_l-i_1)!}
\bigg(M-\sum_{l=2}^d t_l+a-\sum_{l=2}^d i_l-t_1-i_1\bigg)^{M-\sum_{l=2}^d i_l-i_1}\\
&=
\sum_{\substack{0\leq i_2,\ldots, i_d\leq M,\\ \sum_{l=2}^d i_l \leq M}}
\frac{M!\prod_{l=2}^d (t_l+i_l)^{i_l}}{i_2!\ldots i_d!(M-\sum_{l=2}^d i_l)!}
\frac{\partial}{\partial t_1}f_{(M-\sum_{l=2}^d i_l), (a-\sum_{l=2}^d t_l)}^{(2)}(t_1)\\
&=0,
\end{align*}
since \eqref{claim_induction} holds.
\end{proof}

\subsection{Proof of Lemma 5}
\begin{proof}
Note that it is easy to verify that 
\begin{align}\label{expectation_of_likelihood_ratio_full_par}
\int_{X^k}\dd\mu(z^k)p(z^k|\theta)\log \frac{p(z^k|\overline{z^k})}{p(z^k|\theta)}=\frac{1}{2},
\end{align}
and thus, we omit the calculation. Let $A_k:=\{\hat{\theta}(z^k)\in (-a,a)\}$. Since \eqref{expectation_of_likelihood_ratio_full_par} holds and $\hat{\theta}(z^k)=\overline{z^k}$ for $z^k\in A_k$
\begin{align}
\frac{1}{2}&-\int_{X^k}\dd\mu(z^k)p(z^k|\theta)\log \frac{p(z^k|\hat{\theta}(z^k))}{p(z^k|\theta)} \nonumber \\
&=\int_{X^k}\dd\mu(z^k)p(z^k|\theta)\log \frac{p(z^k|\overline{z^k})}
{p(z^k|\theta)}-\int_{X^k}\dd\mu(z^k)p(z^k|\theta)\log \frac{p(z^k|\hat{\theta}(z^k))}{p(z^k|\theta)} \nonumber \\
&=\int_{A_k^c}\dd\mu(z^k)p(z^k|\theta)
\log \frac{p(z^k|\overline{z^k})}{p(z^k|\hat{\theta}(z^k))}.
\label{assertion_to_be_proved}
\end{align}
Since $p(z^k|\overline{z^k})\geq p(z^k|\hat{\theta}(z^k))$, \eqref{assertion_to_be_proved} is positive. Hence,
\begin{align}\label{bounded_above}
\bigg|\int_{X^k}\dd\mu(z^k)p(z^k|\theta)\log \frac{p(z^k|\hat{\theta}(z^k))}{p(z^k|\theta)}-\frac{1}{2}\bigg| = \int_{A_k^c}\dd\mu(z^k)p(z^k|\theta) \log \frac{p(z^k|\overline{z^k})}{p(z^k|\hat{\theta}(z^k))}.
\end{align}
The integrand in the right-hand side of \eqref{bounded_above}  is
\begin{align*}
\log \frac{p(z^k|\overline{z^k})}{p(z^k|\hat{\theta}(z^k))}=k(\overline{z^k}-\hat{\theta}(z^k))\overline{z^k}-\frac{k}{2}((\overline{z^k})^2-(\hat{\theta}(z^k))^2).
\end{align*}
Since the sample mean $\overline{z^k}$ is normally distributed with mean $\theta$ and variance $1/k$, 
\begin{align*}
\int_{A_k^c} &\dd\mu(z^k)p(z^k|\theta) \log \frac{p(z^k|\overline{z^k})}{p(z^k|\hat{\theta}(z^k))}\\
= {}&\int_a^{\infty} \dd u~\phi(u;\theta,1/k) \frac{k(u-a)^2}{2}+\int_{-\infty}^{-a} \dd u~\phi(u;\theta,1/k) \frac{k(u+a)^2}{2}\\
= {}&\int_{\sqrt{k}(a-\theta)}^{\infty} \dd u~\phi(u;0,1) \frac{(u-\sqrt{k}(\theta+a))^2}{2}
+\int_{-\infty}^{-\sqrt{k}(a+\theta)} \dd u~\phi(u;0,1) \frac{(u-\sqrt{k}(\theta-a))^2}{2}\\
\leq {}&\int_{\sqrt{k}\delta}^{\infty} \dd u~\phi(u;0,1) \frac{(u-\sqrt{k}(\theta+a))^2}{2}
+\int_{-\infty}^{-\sqrt{k}\delta} \dd u~\phi(u;0,1) \frac{(u-\sqrt{k}(\theta-a))^2}{2}\\
= {}&\int_{\sqrt{k}\delta}^{\infty} \dd u~\phi(u;0,1) \frac{(u-\sqrt{k}\theta-\sqrt{k}a)^2}{2}
+\int_{\sqrt{k}\delta}^{\infty} \dd u~\phi(u;0,1) \frac{(u+\sqrt{k}\theta-\sqrt{k}a)^2}{2}\\
\leq {}&\int_{\sqrt{k}\delta}^{\infty} \dd u~\phi(u;0,1) (u^2+k(a^2+\theta^2)).
\end{align*}
By the Lebesgue convergence theorem
\[
\lim_{k\to\infty} \int_{\sqrt{k}\delta}^{\infty} \dd u~\phi(u;0,1) u^2=0.
\]
In addition, since $u/(\sqrt{k}\delta)\geq 1$ for $u\geq \sqrt{k}\delta$,
\[
\int_{\sqrt{k}\delta}^{\infty} \dd u~\phi(u;0,1) k(a^2+\theta^2) 
\leq \int_{\sqrt{k}\delta}^{\infty} \dd u~\phi(u;0,1) \frac{2uka^2}{\sqrt{k}\delta} = \frac{2\sqrt{k}a^2}{\delta}\exp(-k\delta^2/2).
\]
Therefore,
\[
\lim_{k\to\infty}\sup_{\theta\in K}\int_{\sqrt{k}\delta}^{\infty} \dd u~\phi(u;0,1) k(a^2+\theta^2) =0.
\]
By \eqref{bounded_above}, the claim is verified.
\end{proof}

\subsection{Proof of Lemma 6}
\begin{proof}
First, we derive 
\[
\int_{\mathbf{R}^M}\dd y^M p(y^M|\hat{\theta}(x^N,y^M))
=1+\frac{M}{N}\int_{-\frac{aN}{\sqrt{M}}-\frac{1}{\sqrt{M}}\sum_{i=1}^Nx_i}^{\frac{aN}{\sqrt{M}}-\frac{1}{\sqrt{M}}\sum_{i=1}^Nx_i} \dd v~ \phi(v;0,1).
\]
From this equation, we find that the normalizing constant of CNML3 does depend on the value $\sum_{i=1}^N x_i$ (see Remark \ref{remark_minimax_regret}).
	
Let $u:=\sum_{i=1}^N x_i$ and let $v^M:=(v_1,\ldots,v_M)^{\top}$ satisfying
\[
v^M=H y^M,
\]
where $H$ is the $M\times M$ orthogonal matrix of the Helmert transformation. From the definition of $H$, we have
\[
\frac{1}{\sqrt{M}}\sum_{i=1}^M y_i=v_1,\quad \sum_{i=1}^M y_i^2=\sum_{i=1}^M v_i^2.
\]
MLE $\hat{\theta}(x^N,y^M)$ is represented in terms of $u$ and $v_1$:
\begin{align*}
\hat{\theta}(u,v_1):=\left\{
\begin{array}{ll}
-a, & \mathrm{if}~ \frac{u+\sqrt{M}v_1}{N+M}<-a, \\
a, &  \mathrm{if}~ \frac{u+\sqrt{M}v_1}{N+M}>a, \\
\frac{u+\sqrt{M}v_1}{N+M}, & \mathrm{otherwise}.
\end{array}
\right.
\end{align*}
Because $H$ is orthogonal,
\begin{align*}
\int_{\mathbf{R}^M}&\dd y^M p(y^M|\hat{\theta}(x^N,y^M))\\
={}&\int_{\mathbf{R}^M}\dd v^M \frac{1}{(2\pi)^{\frac{M}{2}}} \exp\bigg(-\frac{1}{2}\sum_{i=2}^Mv_i^2-\frac{1}{2}(v_1-\sqrt{M}\hat{\theta}(u,v_1))^2\bigg)\\
={}&\int_{\mathbf{R}}\dd v_1\frac{1}{\sqrt{2\pi}}\exp\bigg(-\frac{1}{2}(v_1-\sqrt{M}\hat{\theta}(u,v_1))^2\bigg)\\
={}&\int_{-\infty}^{-\frac{a(N+M)}{\sqrt{M}}-\frac{u}{\sqrt{M}}} \dd v_1 \frac{\exp\big(-\frac{(v_1+\sqrt{M}a)^2}{2}\big)}{\sqrt{2\pi}}
+\int_{\frac{a(N+M)}{\sqrt{M}}-\frac{u}{\sqrt{M}}}^{\infty} \dd v_1 \frac{\exp\big(-\frac{(v_1-\sqrt{M}a)^2}{2}\big)}{\sqrt{2\pi}}\\
&+\int_{-\frac{a(N+M)}{\sqrt{M}}-\frac{u}{\sqrt{M}}}^{\frac{a(N+M)}{\sqrt{M}}-\frac{u}{\sqrt{M}}} \dd v_1 \frac{1}{{\sqrt{2\pi}}}\exp\bigg(-\frac{(Nv_1-\sqrt{M}u)^2}{2(N+M)^2}\bigg)\\
={}&1+\frac{M}{N}\int_{-\frac{aN}{\sqrt{M}}-\frac{u}{\sqrt{M}}}^{\frac{aN}{\sqrt{M}}-\frac{u}{\sqrt{M}}} \dd v_1 \phi(v_1;0,1).
\end{align*}
Next, we verify that condition A5 is satisfied.
\begin{align*}
\frac{M}{N}\int_{-\frac{aN}{\sqrt{M}}-\frac{u}{\sqrt{M}}}^{\frac{aN}{\sqrt{M}}-\frac{u}{\sqrt{M}}} \dd v_1~ \phi(v_1;0,1)
&=\frac{M}{N}\int_{-\frac{aN}{\sqrt{M}}}^{\frac{aN}{\sqrt{M}}} \dd v_1~ \phi(v_1;0,1)\exp\bigg(\frac{u}{\sqrt{M}}v_1-\frac{u^2}{2M}\bigg)\\
&\leq \exp\bigg(\frac{aN|u|}{M}\bigg)\frac{M}{N}\int_{-\frac{aN}{\sqrt{M}}}^{\frac{aN}{\sqrt{M}}} \dd v_1~ \phi(v_1;0,1).
\end{align*}
Since $\exp(aN|u|/M)\geq 1$,
\[
1+\frac{M}{N}\int_{-\frac{aN}{\sqrt{M}}-\frac{u}{\sqrt{M}}}^{\frac{aN}{\sqrt{M}}-\frac{u}{\sqrt{M}}} \dd v_1~ \phi(v_1;0,1)
\leq \exp\bigg(\frac{aN|u|}{M}\bigg)\bigg(1+\frac{M}{N}\int_{-\frac{aN}{\sqrt{M}}}^{\frac{aN}{\sqrt{M}}} \dd v_1~ \phi(v_1;0,1)\bigg).
\]
Similarly, we find the lower bound as follows:
\[
1+\frac{M}{N}\int_{-\frac{aN}{\sqrt{M}}-\frac{u}{\sqrt{M}}}^{\frac{aN}{\sqrt{M}}-\frac{u}{\sqrt{M}}} \dd v_1~ \phi(v_1;0,1)
\geq \exp\bigg(-\frac{aN|u|}{M}-\frac{u^2}{2M}\bigg)\bigg(1+\frac{M}{N}\int_{-\frac{aN}{\sqrt{M}}}^{\frac{aN}{\sqrt{M}}} \dd v_1~ \phi(v_1;0,1)\bigg).
\]
Therefore,
\begin{align*}
\bigg|\log\int_{\mathbf{R}^M}\dd y^M p(y^M|\hat{\theta}(x^N,y^M))-\log\bigg(1+\frac{M}{N}\int_{-\frac{aN}{\sqrt{M}}}^{\frac{aN}{\sqrt{M}}} \dd v_1~ \phi(v_1;0,1)\bigg)\bigg|\leq \frac{aN|u|}{M}+\frac{u^2}{2M}.
\end{align*}
From this inequality, if we set 
\[
C^{N,M}=\log\bigg(1+\frac{M}{N}\int_{-\frac{aN}{\sqrt{M}}}^{\frac{aN}{\sqrt{M}}} \dd v_1~ \phi(v_1;0,1)\bigg),
\]
then the claim is verified because $\E_{\theta}[|u|]$ and $\E_{\theta}[u^2]$ is uniformly bounded in $\theta\in K$ and do not depend on $M$.
\end{proof}

\bibliography{ref_nml_metr.bib}
\end{document}